\documentclass[11pt]{article}
\usepackage{multicol,amsmath,amssymb,wasysym,mathtools,graphicx,tikz,geometry,hyperref,amsthm,makeidx,dsfont,sseq,pagecolor,xcolor,enumerate,pdflscape,spectralsequences,pdfpages}
\usepackage[normalem]{ulem}
\usepackage{eucal}
\usepackage{kpfonts} 
\usepackage{url}
\usetikzlibrary{cd,automata,matrix,decorations.pathreplacing,decorations.markings}

\definecolor{wongorange}{RGB}{230,159,0}
\definecolor{wonglightblue}{RGB}{86,180,233}
\definecolor{wonggreen}{RGB}{0,158,115}
\definecolor{wongyellow}{RGB}{240,228,66}
\definecolor{wongdarkblue}{RGB}{0,114,178}
\definecolor{wongred}{RGB}{213,94,0}
\definecolor{wongpink}{RGB}{204,121,167}
\definecolor{scottygreenyellow}{RGB}{169,255,23}
\definecolor{vantablack}{RGB}{26,26,26}

\hypersetup{
    colorlinks=true,
    linkbordercolor={vantablack},
    linkcolor={black},
    urlcolor = {black},
    citecolor = {black}}

\usepackage[textwidth=2cm,textsize=tiny,colorinlistoftodos]{todonotes}
\newcommand{\stnote}[1]{}

\newcommand{\by}{\times}
\newcommand{\oby}{\otimes}
\newcommand{\sets}[1]{\left\{ #1 \right\}}
\newcommand{\abs}[1]{\left| #1 \right|}
\newcommand{\aabs}[1]{\left\| #1 \right\|}
\newcommand{\parenths}[1]{\left( #1 \right)}

\newcommand{\anglebracket}[1]{\left\langle #1 \right\rangle}
\newcommand{\ol}[1]{\overline{#1}}
\newcommand{\ul}[1]{\underline{#1}}
 
\newcommand{\inv}{^{-1}}

\renewcommand{\phi}{\varphi}

\newcommand{\xr}[1]{\xrightarrow{#1}}

\newcommand{\into}{\hookrightarrow}

\newcommand{\downto}{\downarrow}

\newcommand{\ds}{\displaystyle}
\newcommand{\eps}{\varepsilon}
\newcommand{\dee}{\partial}
\renewcommand{\tilde}[1]{\widetilde{#1}}
\renewcommand{\hat}[1]{\widehat{#1}}

\newcommand{\acts}{\curvearrowright}

\renewcommand{\implies}{\Rightarrow}
\newcommand{\nond}{\noindent}
\newcommand{\Fr}{\mathsf{Fr}}
\newcommand{\mCP}{\mathbb{CP}}
\newcommand{\mRP}{\mathbb{RP}}
\newcommand{\mHP}{\mathbb{HP}}
\newcommand{\SO}{\mathsf{SO}}
\newcommand{\Spin}{\mathsf{Spin}}

\newcommand{\Mat}{\mathsf{Mat}}

\newcommand{\lU}{\mathsf{U}}
\newcommand{\Pin}{\mathsf{Pin}}
\newcommand{\mCl}{\mathbb{C}\text{l}}
\newcommand{\Sq}{\text{Sq}}
\newcommand{\Sp}{\mathsf{Sp}}

\newcommand{\uno}{\text{id}}

\newcommand{\Diff}{\text{Diff}}
\newcommand{\Homeo}{\text{Homeo}}
\newcommand{\Emb}{\text{Emb}}
\newcommand{\K}{\text{K}}

\newcommand{\mtR}{\widetilde{\mR}}
\newcommand{\mtC}{\widetilde{\mathbb{C}}}
\newcommand{\mtCP}{\widetilde{\mathbb{CP}}}

\newcommand{\SU}{\mathsf{SU}}

\DeclareMathOperator{\proj}{\text{proj}}

\DeclareMathOperator*{\colim}{\mathop{\text{colim}}} 

\DeclareMathOperator{\Th}{\text{Th}}

\DeclareMathOperator{\End}{\text{End}}
\DeclareMathOperator{\Hom}{\text{Hom}}
\DeclareMathOperator{\FBF}{\text{FBF}}
\DeclareMathOperator{\BF}{\text{BF}}
\DeclareMathOperator{\SW}{\text{SW}}
\DeclareMathOperator{\SWhat}{\hat{\SW}}
\DeclareMathOperator{\SWapp}{\SW_{\approx}}
\DeclareMathOperator{\SWappsph}{\SW_{\approx}^{+}}
\DeclareMathOperator{\FSWhat}{\hat{\text{FSW}}}
\DeclareMathOperator{\FSWhatappthom}{\hat{\text{FSW}}_{\approx}^{\text{Th}}}

\DeclareMathOperator{\ctheta}{c_{\theta}}
\DeclareMathOperator{\cthetabar}{c_{\ol{\theta}}}
\DeclareMathOperator{\Res}{\text{Res}}

\newcommand{\sectionlabel}[1]{\section{#1}\label{sec:#1}}
\newcommand{\subsectionlabel}[1]{\subsection{#1}\label{subsec:#1}}
\newcommand{\subsubsectionlabel}[1]{\subsubsection{#1}\label{subsubsec:#1}}
\newcommand{\Hmatrix}{\begin{psmallmatrix}0&1\\1&0\end{psmallmatrix}}
\newcommand{\Pintwo}{{\Pin(2)}}

\newcommand{\lineyspace}{
  \bigskip
  \hrule
  \bigskip
}
\newcommand{\dirac}{\partial\!\!\!\!/}

\makeatletter
\newcommand*{\relrelbarsep}{.386ex}
\newcommand*{\relrelbar}{%
  \mathrel{%
    \mathpalette\@relrelbar\relrelbarsep
  }%
}
\newcommand*{\@relrelbar}[2]{%
  \raise#2\hbox to 0pt{$\m@th#1\relbar$\hss}%
  \lower#2\hbox{$\m@th#1\relbar$}%
}
\providecommand*{\rightrightarrowsfill@}{%
  \arrowfill@\relrelbar\relrelbar\rightrightarrows
}
\providecommand*{\leftleftarrowsfill@}{%
  \arrowfill@\leftleftarrows\relrelbar\relrelbar
}
\providecommand*{\xrightrightarrows}[2][]{%
  \ext@arrow 0359\rightrightarrowsfill@{#1}{#2}%
}
\providecommand*{\xleftleftarrows}[2][]{%
  \ext@arrow 3095\leftleftarrowsfill@{#1}{#2}%
}
\makeatother

\makeatletter
\@tfor\next:=ABCDEFGHIJKLMNOPQRSTUVWXYZ\do{%
  \def\command@factory#1{%
    \expandafter\def\csname cal#1\endcsname{\mathcal{#1}}
  }
 \expandafter\command@factory\next
}
\makeatletter
\@tfor\next:=ABCDEFGHIJKLMNOPQRSTUVWXYZ\do{%
  \def\command@factory#1{%
    \expandafter\def\csname m#1\endcsname{\mathbb{#1}}
  }
 \expandafter\command@factory\next
}
\makeatletter
\@tfor\next:=ABCDEFGHIJKLMNOPQRSTUVWXYZabcdefghijklmnopqrstuvwxyz\do{%
  \def\command@factory#1{%
    \expandafter\def\csname fr#1\endcsname{\mathfrak{#1}}
  }
 \expandafter\command@factory\next
}
\makeatletter
\@tfor\next:=ABCDEFGHIJKLMNOPQRSTUVWXYZabcdefghijklmnopqrstuvwxyz\do{%
  \def\command@factory#1{%
    \expandafter\def\csname bf#1\endcsname{\mathbf{#1}}
  }
 \expandafter\command@factory\next
}
\makeatletter
\@tfor\next:=ABCDEFGHIJKLMNOPQRSTUVWXYZabcdefghijklmnopqrstuvwxyz\do{%
  \def\command@factory#1{%
    \expandafter\def\csname v#1\endcsname{\vec{#1}}
  }
 \expandafter\command@factory\next
}
\makeatletter
\@tfor\next:=ABCDEFGHIJKLMNOPQRSTUVWXYZabcdefghijklmnopqrstuvwxyz\do{%
  \def\command@factory#1{%
    \expandafter\def\csname sf#1\endcsname{\mathsf{#1}}
  }
 \expandafter\command@factory\next
}

\makeatletter
\@tfor\next:=ABCDEFGHIJKLMNOPQRSTUVWXYZabcdefghijkmnopqstuvwxyz\do{%
  \def\command@factory#1{%
    \expandafter\def\csname x#1\endcsname{\text{#1}}
  }
 \expandafter\command@factory\next
}

\newtheoremstyle{boldtitle}      
  {3pt}                          
  {3pt}                          
  {\normalfont}                  
  {}                             
  {\bfseries}                    
  {.}                            
  {.5em}                         
  {} 

\theoremstyle{plain}
\newtheorem{theorem}{Theorem}[section]
\newtheorem*{theorem*}{Theorem}
\newtheorem{corollary}[theorem]{Corollary}
\newtheorem*{corollary*}{Corollary}
\newtheorem{lemma}[theorem]{Lemma}
\newtheorem{proposition}[theorem]{Proposition}

\theoremstyle{definition}

\newtheorem{example}[theorem]{Example}

\theoremstyle{boldtitle}
\newtheorem{remark}[theorem]{Remark}

\makeindex

\title{The Boundary Dehn Twist on a Punctured Connected Sum of Two K3 Surfaces is Nontrivial in the Smooth Mapping Class Group}
\author{Scotty Tilton}

\definecolor{gold}{HTML}{ffd700}

\begin{document}
\maketitle
\begin{abstract}
We prove that the boundary Dehn twist on $\K3\#\K3\setminus B^4$ is nontrivial in the smooth mapping class group, providing another example of an exotic diffeomorphism on a simply-connected spin four-manifold. 
We do so by finding an algebraic criterion which must be satisfied if the two maps are smoothly isotopic. The main tools involved are the $\Pintwo$-equivariant families Bauer-Furuta invariant, equivariant topological $\sfK$-theory, and the Atiyah-Hirzebruch spectral sequence to show this algebraic criterion cannot be satisfied, and this establishes the result. As a corollary, we find any smooth bundle $\K3\#\K3\into E\downto S^2$ has $w_2(T^vE)=0$, so $E$ is spin.
\end{abstract}


\sectionlabel{Introduction}

Given a (smooth) manifold $X$, one can try to understand its (smooth) mapping class group. The continuous mapping class group is $$\text{MCG}^{\text{top}}(X):=\pi_0(\Homeo(X),\uno)\text{ or }\text{MCG}^{\text{top}}_\dee(X,\dee X):=\pi_0(\Homeo_\dee(X,\dee X),\uno)$$ and smoothly it is $$\text{MCG}^{\text{smth}}(X):=\pi_0(\Diff(X),\uno)\text{ or }\text{MCG}^{\text{smth}}_\dee(X,\dee X):=\pi_0(\Diff_\dee(X,\dee X),\uno).$$ Here, $\Homeo_\dee$ and $\Diff_\dee$ denote those homeomorphisms and diffeomorphisms which are the identity in a neighborhood of the boundary. The topology on $\Homeo(X)$ is the compact open topology, and the topology of $\Diff(X)$ is the $C^\infty$ topology. 

The relationships between different mapping class groups can change as the dimension of the manifold changes. For example, in dimensions 1-3, the homeomorphism, diffeomorphism group, and the piecewise linear group are equivalent. This begins to change in dimension 4, with some topological manifolds, such as Freedman's $E_8$-manifold, not admitting a smooth structure at all. The Kirby-Siebenman invariant is an obstruction to lifting a topological structure to a PL structure in dimension $\geq 5$. 
An overarching goal is to understand these groups and their relationships between one another. Our focus in this paper is on understanding a specific difference between the continuous and smooth mapping class groups of the simply-connected spin four-manifold $\K3\#\K3\setminus B^4$.

A key ingredient is Quinn's theorem and some analogues which provide a relatively easy criterion to check whether a homeomorphism is continuously isotopic to the identity. Let $X$ be a closed, simply-connected four-manifold. The theorem states: If $\phi:X\to X$ is a homeomorphism which induces the identity on homology, then $\phi$ is in the identity component of $\Homeo(X)$ \cite{QuinnIsotopy86}. Orson-Powell \cite[Thm. E]{orson2025mapping} and Krannich-Kupers \cite[Lem. 2]{krannich2024torelligroupsdehntwists} generalized this result to simply-connected manifolds with boundary. The boundary Dehn twist $\delta$, described in Section \ref{subsec:Dehn Twists}, is such a homeomorphism.  

Now, if some diffeomorphism $h$ is topologically isotopic to the identity, can we determine whether it is smoothly isotopic to the identity? If it is not, we call the diffeomorphism $h$ an \textbf{exotic diffeomorphism}. Another way to express those exotic diffeomorphisms is to determine whether $[h]$ is a nontrivial element in $$\ker\bigg(\text{MCG}^{\text{smth}}(X)\to \text{MCG}^{\text{top}}(X)\bigg).$$

The first known example of an exotic diffeomorphism was found by Ruberman in \cite{ruberman1998obstruction} using Donaldson-type invariants. Since then, several authors have discovered other exotic diffeomorphisms, such as the boundary Dehn twist on $\K3\setminus B^4$ by \cite{Baraglia_2022,KM20}, the boundary Dehn twist on $S^2\by S^2\#\K3\setminus B^4$ \cite{JLin20}, and the boundary Dehn twist on some irreducible four-manifolds in \cite{baraglia2024irreducible4manifoldsadmitexotic}. Recently, some authors found exotic diffeomorphisms which survived multiple stabilizations in \cite{kang2025exoticdiffeomorphismscontractible4manifold}. Other authors studying phenomena in this realm include, but are not limited to, \cite{kang2024exoticdehntwistshomotopy}, \cite{konno2024monodromydiffeomorphismweightedsingularities}, and \cite{konno2024exoticdehntwists4manifolds}. 

One can generalize the boundary Dehn twist, and these generalizations are the study of many other authors. One common theme among these results and research directions is using invariants derived from the Seiberg-Witten map, namely the (equivariant) Bauer-Furuta invariant and families versions of it. 

\begin{remark}
  Something to note is that for non-spin four-manifolds, the boundary Dehn twist is trivial, i.e., not exotic \cite[Cor. A.5]{orson2025mapping}. This is why we focus on spin manifolds, since these manifolds \emph{could} have exotic boundary Dehn twists. 
\end{remark}

\subsectionlabel{Main Theorem and Outline of Proof}
Our main tool is the families version of the $\Pintwo$-equivariant Bauer-Furuta invariant which requires tools from equivariant stable homotopy theory. This paper records the following result in smooth topology.   

\begin{theorem}\label{thm:main-thm}
    The boundary Dehn twist is an exotic diffeomorphism of $\K3\#\K3\setminus B^4$. 
\end{theorem}

This is proved by contradiction. Using the generalized theorem of Quinn, we know that $\delta,\uno$ are topologically isotopic. Suppose there is in fact a smooth isotopy between the identity and the boundary Dehn twist. If this were true, we show the families Bauer-Furuta invariant for the product family with the product spin structure is equal to the invariant for the product family with the twisted spin structure. With this hypothetical fact, we are able to 
define a Seiberg-Witten-like map between Hilbert bundles over $S^2$ by using the stable homotopy between the family Bauer-Furuta invariants of the two spin families over the product spin structure. We then observe cofiber sequences of the spaces involved and our hypothetical Bauer-Furuta-type map, and we apply $C_2$-equivariant $\sfK$-theory. This enforces an algebraic requirement which would be satisfied by such a bundle and bundle maps. We then show such an algebraic condition cannot be satisfied by explicitly calculating the $\sfK_{C_2}$-groups of the spaces involved. Therefore, such a bundle map cannot exist, and hence our hypothetical isotopy cannot exist, proving the Dehn twist is exotic. A corollary of this result is the following.

\begin{corollary}\label{cor:main-cor}
  A smooth fiber bundle $\K3\#\K3\into E\downto S^2$ must have $w_2(T^vE)=0$. Equivalently, for any smooth fiber bundle $X\into E\downto S^2$, the total space $E$ is spin. 
\end{corollary}

\subsectionlabel{Future Directions}
Since there has been study of $\K3$ in \cite{Baraglia_2022,KM20}, $\K3\# S^2\by S^2$ in \cite{JLin20}, and now $\K3\#\K3$ in this paper, the next natural simply-connected, spin four-manifold to study is the elliptic surface $E(4)$, diffeomorphic to $\K3\#\K3\# S^2\by S^2$ by the properties laid out by \cite[Prop. 3.1.11]{Gompf-Stip-Kirby-Book} combined with the classification of simply-connected four-manifolds by \cite{Freedman82}. A construction of the elliptic manifolds can be found in \cite[\S3.1,\S7.3]{Gompf-Stip-Kirby-Book}. Note that $\K3$ is diffeomorphic to $E(2)$, and the even elliptic surfaces $E(2n)$ are all spin, where the odd elliptic surfaces are not. 

The main theorem of this paper and former results leads us to ask: is the Dehn twist isotopic to the identity after a single stabilization on a punctured $\K3\#\K3?$ Is it exotic on other even elliptic surfaces? Lin's result was the first instance of an exotic diffeomorphism surviving a single stabilization, and the recent result of Kang-Park-Taniguchi \cite{kang2025exoticdiffeomorphismscontractible4manifold} showed that an exotic diffeomorphism can survive two stabilizations. 

The key difference between $E(4)$ and these other examples is in their intersection form. If we choose to follow the $\sfK$-theoretic route, there is a good indication that we will have to work with $\sfK\sfO$-theory in some cases. A table is provided with certain spin four-manifolds, their intersection forms, the $\Pintwo$-equivariant stable homotopy group corresponding to the Bauer-Furuta invariant, and whether $\sfK$-theory or $\sfK\sfO$-theory would make more sense. This is because $S^{2n\mtR}\cong S^{n\mtC}$ is a complex representation sphere, whereas $S^{(2n+1)\mtR}$ is not.  In other words, the representations involved, when viewed as complex $\Pintwo$-representations, work well in $\sfK$-theory, whereas representations which must be witnessed as real $\Pintwo$-representations align themselves better with $\sfK\sfO$-theoretic methods.
\begin{center}
  \begin{tabular}{c|c|c|c}
    Manifold & Intersection Form & Homotopy Group & ``$\sfK$''-theory\\\hline
    $\K3$ & $2E_8\oplus 3\Hmatrix$ & $\pi^\Pintwo_{\mH}(S^{3\mtR})$ & $\sfK\sfO$\\\hline
    $\K3\# S^2\by S^2$ & $2E_8\oplus 4\Hmatrix$ & $\pi^\Pintwo_{\mH}(S^{4\mtR})$ & $\sfK$\\\hline
    $\K3\#\K3$ & $4E_8\oplus 6\Hmatrix$ & $\pi^\Pintwo_{2\mH}(S^{6\mtR})$ & $\sfK$\\\hline
    $\K3\#\K3 \# S^2\by S^2$ & $4E_8\oplus 7\Hmatrix$ & $\pi^\Pintwo_{2\mH}(S^{7\mtR})$ & $\sfK\sfO$\\
  \end{tabular}
\end{center}
An interesting avenue would be to apply the methods of this paper to some of the prior results find alternative methods of proof and perhaps a generalization.

More questions about these exotic phenomena include: 
\begin{itemize}
  \item How many stabilizations are enough to remove exotica? 
  \item What is the structure of $\pi_k(\Diff(X))$ for different $k$ and $X$? Some authors have been exploring this avenue, such as in \cite{Baraglia2023NontrivialK3}.
  \item What information can we determine about $\pi_k(\Diff(X))$? For $k=0$, the abelianization is congruent to $H_1(B\Diff(X))$, and it has been found that for some manifolds, this group is infinite rank such as in \cite{auckly2025familiesdiffeomorphismsembeddingspositive,Konno_2024}. 
\end{itemize}

\subsectionlabel{Outline of Paper} 
The paper is organized as follows. 

In \S2, we discuss the important background material and topological spaces involved in the proof. This section is more general than the later sections. Specifically, we discuss Dehn twists, the Seiberg-Witten map, the Bauer-Furuta invariant, the families Bauer-Furuta invariant, their equivariance, and finally, we cover some important constructions used in the paper and important actions of $\Pintwo$ on some basic spaces. 

In \S3, we bring in the spaces that are the focus of the paper, some facts about these spaces, and state the key hypothesis needed in the proof by contradiction. 

In , we explicitly describe the family Bauer-Furuta invariants for two mapping tori. We then construct a map between bundles over $S^2$, and break down the spaces involved into cofiber sequences and a commuting diagram. After applying $\sfK$-theory to the diagram, we inherit an algebraic criterion to check. This section ends with a desuspension result, allowing us to work with simpler spaces in our calculations. We proceed to calculate the $\sfK$-theory of key spaces involved in the commutative diagram.

In \S5 we prove the main result by finding a contradiction based on our initial faulty hypothesis, stated at the end of \S3. We also find a result on the second Stiefel-Whitney class of families of spin four-manifolds over $S^2$.

\subsectionlabel{Notation}
    \begin{itemize}
      \item $\mH = \mC\oplus \mC j$.
      \item $\mtR$ is $\mR$ with the $\Pintwo$-action generated by $e^{i\theta}\mapsto \uno$ and $j\mapsto -\uno$. 
      \item $\mtC$ is $\mC$ with the $\Pintwo$-action generated by $e^{i\theta}\mapsto \uno$ and $j\mapsto -\uno$. 
      \item $\mtCP^n$ denotes $\mCP^n$ with the left action of $C_2$ given by $$[\alpha_0:\beta_0:\cdots \alpha_n:\beta_n]\mapsto[-\ol{\beta}_0:\ol{\alpha}_0\cdots -\ol{\beta}_n:\ol{\alpha}_n].$$
      \item For $G$ a group and $Y$ a left (or right) $G$-space, we denote the quotient on the left (or right) by $_{G\backslash }Y$ (or $Y_{/G}$).
      \item We may abuse notation and write $\mtR$ and $\mtC$ as the signed $C_2$-representations, $\sigma_\mR$ and $\sigma_\mC$, as well. 
      \item We let $\sfR(G)$ and $\sfR\sfO(G)$ denote the complex and real representation rings, respectively, of $G$ representations. 
      \item If $V$ is a vector space, $S(V):=\sets{v\in V\mid \aabs{v}=1}.$ 
      \item $S^V$ denotes the one-point compactification of $V$.
      \item For a based space $Y$, we write $\Sigma^V Y := S^V\wedge Y.$
      \item For a vector space $V$, we use $V^{\oplus n}$ to denote iterated direct sum $n$ times and similar notation for $\otimes, \wedge, $ et cetera. 
      \item $nV:=V^{\oplus n}.$ 
      \item We denote the trivial vector bundle over a base $B$ with fiber $V$ as $\ul{V} :=V\by B$. 
      \item If $Y$ is a topological space which is not a vector space, $S(Y)$ denotes the unreduced suspension of $Y$. 
      \item $Y_+$ denotes $Y\sqcup \sets{\ast}$ as a based space with basepoint $\ast$. 
      \item When $Y$ is a CW-complex, $Y^{(k)}$ denotes the $k$-skeleton of $Y$.
      \item If $X$ is a simply-connected smooth four-manifold, $X^\circ:= X\setminus B^4$.
      \item $\delta_X^\circ$ denotes the boundary Dehn twist on $X^\circ$, unless the space $X^\circ$ is clear from context. 
      \item The Clifford algebra $\text{Cl}(n)$ is the quotient of the tensor algebra $T(\mR^n):= \bigoplus_{k=0}^\infty (\mR^n)^{\otimes k}$ by the two-sided ideal generated by $v\otimes v +\aabs{v}^2$ for $v\in \mR^n$. 
      \item $\mCl(n)$ denotes the complexified Clifford algebra $\text{Cl}(n)\otimes_\mR \mC$. 
      \item $\pi_n^s := \colim_k \pi_{n+k}(S^k)$, the $n$-th stable homotopy group of spheres. 
      \item Let $\calU$ be a complete Grothendieck universe of $G$-representations. We denote stable homotopy classes of maps between based $G$-spaces $Y,Z$ as $\sets{X,Y}^G :=\underset{\underset{\text{fin. dim.}}{V\subset \calU}}{\colim}[\Sigma^V Y,\Sigma^V Z]^G.$
      \item For a representation spheres $S^V$, we will write $\pi_V^G(Y) :=\sets{S^V,Y}^G$, the $G$-equivariant stable homotopy group of $Y$.
    \end{itemize}

\subsectionlabel{Acknowledgements}

I would like to acknowledge my coadvisors, Zhouli Xu (UCLA) and Jianfeng Lin (Tsinghua University), for their patience and assistance with the background theory and the opportunities for collaboration they introduced me to. I am indebted to Dan Isaksen, whose eCHT graduate student fellowship allowed me the time and resources to learn and research successfully. I am grateful to the University of California, San Diego for my time as a Ph.D. student in ``America's Finest City.'' I am extremely grateful to my friends, family, and partner, Martha, for their patience and support throughout the years. This project was partially supported by NSF Grant 2135884.

\sectionlabel{Background}

\subsectionlabel{Dehn Twists}
Given an oriented $n$-manifold $M$ and an embedded $S^{n-1}\into M$, we can define a Dehn twist as follows. Pick a loop $\gamma\in \pi_1(\SO(n),\uno)$ and trivialize a closed tubular neighborhood of the embedded $S^{n-1}$, $$\ol{D(\nu S^{n-1})}\cong [0,1]\by S^{n-1}.$$ 
If we apply the map 
\begin{center}
  \begin{tikzcd}[ampersand replacement = \&]
    \delta_\gamma\colon I\by S^{n-1}\ar[r]\& I\by S^{n-1}\\ [-20pt]
    (t,x)\ar[r,mapsto] \& (t,\gamma(\rho(t))\cdot x)
  \end{tikzcd}
\end{center} to $D(\nu(S^{n-1}))$ and the identity elsewhere on $M$, this is a general version of a boundary Dehn twist. Here $\rho(t):I\to I$ is a monotonically increasing smooth function which is 0 in a neighborhood of 0 and 1 in a neighborhood of 1. 
This construction matches many of the usual definitions we have for Dehn twists. In this paper, we focus on dimension $n=4$.

\begin{remark}
  There are several other Dehn twists. An example comes from an embedding of $S^1$ into a manifold of codimension greater than one. Take a knot, $S^1\into S^3$ for example, and then, since the tubular neighborhood is $S^1\by D^2$, we use the $S^1$ coordinate to rotate the $D^2$ coordinate by an element of $\SO(2)$. 
  
  Another example is the Dehn-Seidel twist on $T^\ast S^2$, which has a $180^\circ$ twist on the zero section and is the identity far away from the zero section. 
  
  The last example we mention is a different sort which comes from manifolds with boundaries that have a natural $S^1$-action, such as a four-manifold whose boundary is a Milnor fibration.
\end{remark}\lineyspace

To be explicit, in this paper, the \textbf{Dehn twist along the boundary} of a four-manifold $(X,\dee X)$ with $\dee X\cong S^3$, is the diffeomorphism $\delta\colon X\to X$ which is the identity near the boundary and on the interior of the manifold, and in a collar neighborhood, $N$, is modeled on 
\begin{center}
  \begin{tikzcd}[ampersand replacement = \&]
    \delta'\colon S^3\by I\ar[r]\& S^3\by I \\ [-20pt]
    ((z,w),t)\ar[r,mapsto] \& ((e^{2\pi i \rho(t)}z, w), t)
  \end{tikzcd}
\end{center} where $\rho(t)\colon I\to I$ is a monotonically increasing smooth function which is 0 in a neighborhood of 0 and 1 in a neighborhood of 1. A cartoon schematic for the Dehn twist can be found in figure \ref{fig:dehn-cartoon}. 

\begin{figure}[h!]\label{fig:dehn-cartoon}
  \begin{center}
  \begin{tikzpicture}
    \def \r {1.5}
    \draw[gray, ball color = gray,shading =ball,opacity=.5](15:\r) arc (15:345:\r);
    \draw[thick,gray] (15:{2.2*\r}) arc (90:-90:.05 and {sin(15)*2.2*\r});
    \draw[thick,gray] (15:{1.7*\r}) arc (90:-90:.05 and {sin(15)*1.7*\r});
    \draw[thick,gray,dashed] (15:{2.2*\r}) arc (90:450:.05 and {sin(15)*2.2*\r});
    \draw[thick,gray,dashed] (15:{1.7*\r}) arc (90:450:.05 and {sin(15)*1.7*\r});
    \draw[gray,ball color = gray, shading =ball,opacity  = .5] (15:\r) -- (15:{2.2*\r}) arc (90:-90:.05 and {sin(15)*2.2*\r}) -- (-15:\r)  arc (195:165:\r);
    \draw[gray] (15:{2.2*\r}) -- (15:\r);
    \draw[gray] (-15:{2.2*\r}) -- (-15:\r);
    \draw[gray] (-15:\r) arc (195:165:\r);
    \draw[wongred,thick] (0:{.93*\r}) to[bend left] (15:{1.2*\r});
    \draw[wongred,dashed,thick] (15:{1.2*\r}) to[out = 315,in = 150] (-15:{1.5*\r});
    \draw[wongred,thick] (0:{1.68*\r}) to[bend left] (-15:{1.5*\r});
    \draw (0,0) node {$X$};
    \draw (0:{2.35*\r}) node {$\dee X$};
    \draw[decorate,decoration= {brace,raise = 5pt},thick] (-15:{2.2*\r}) -- (-15:\r) node[midway,below = 6pt] {$N$};
  \end{tikzpicture}
  \end{center}
  \caption{Cartoon for $(X,\dee X)$ with collar neighborhood $N$. The support of the Dehn twist contains a red line giving an indication of where the diffeomorphism is not the identity.}
\end{figure}
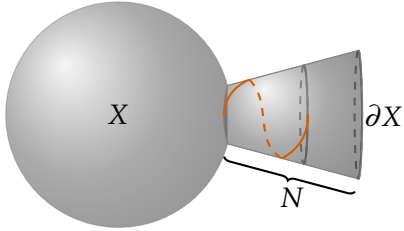

\begin{remark}\label{rmk:dehn-twist-in-image-of-LES}
One may ask why the boundary Dehn twist is showing up in $\text{MCG}_\dee^{\text{smth}}(X,\dee X)$. This is based on some homotopy calculations of embeddings of submanifolds into $X$. 

Let $\Sigma\into X$ be a smooth embedding of a submanifold. There is a fibration $\Diff(X)\to \Emb(\nu\Sigma, X)$ between diffeomorphisms on $X$ to embeddings of the normal bundle of $\Sigma$ into $X$ with homotopy fiber $\Diff_{\dee}(X\setminus \nu\Sigma)$. 

One example is when we take $\Sigma$ to be a point and we take $X$ to be a simply-connected, spin four-manifold. For any embedding of $\ast\into X$, we find $\Emb(\nu\ast,X)\simeq \Fr(X)$ and $\pi_1(\Fr(X),\theta) \cong \mZ/2$ for some chosen basepoint $\theta$ in the frame bundle of $X$.

From the homotopy fibration sequence, we obtain a long exact sequence in homotopy groups $$\cdots\to \pi_1(\Diff(X),\uno)\to \pi_1(\Fr(X),\theta)\to \pi_0(\Diff_\dee(X\setminus B^4),\uno)\to \cdots.$$ The image of the nontrivial loop $1\in \pi_1(\Fr(X),\theta)$ is the boundary Dehn twist on $X\setminus B^4$. This is where the boundary Dehn twist arises in the mapping class group. Its relationship to our problem can be read about in \cite[\S2]{KM20} and \cite{Baraglia_2022}. 
\end{remark}

\subsectionlabel{Seiberg-Witten to family Bauer-Furuta}
\subsubsectionlabel{Seiberg-Witten Map}
We collect several results and observations from Morgan's book \cite[Ch 2,3]{MorganSWBook}, though we would also like to point out the great exposition in Fushida-Hardy's course notes from Manolescu's four-manifold course at Stanford \cite[Ch 3,4]{shintaro-notes-manolescu-283A-4-manifolds}.

Let $(X,\frs)$ be a simply-connected spin four-manifold with spin$^c$ structure $\frs$ obtained from the spin structure on $X$ and a spin-invariant metric $g$. Since $\Spin^c(n) = \Spin(n)\by_{\sets{\pm 1}}\lU(1)$, 
we replace all the fibers of the spin structure with $\Spin^c(4)$ to get the spin$^c$ structure, $\frs$, on $X$. 

Let $S: = \frs\by_{\Spin(4)} \mC^4.$ It is worth noting that $\mCl(4)$ is isomorphic to $\Mat_{4\by 4}(\mC)$ as $\mC$-algebras. Identifying $\Spin(4)$ as the subgroup $\Spin(4)\by_{\sets{\pm 1}}\sets{\pm 1}\subset \Spin^c(4)$, leads to an action of $\Spin(4)$ on $\mC^4$. 
This action yields an irreducible complex representation which extends to $\mCl(4)$. 
Set $$\omega_\mC:= -e_1e_2e_3e_4\in \Spin^c(4).$$ 
Observe that $\omega_\mC\in \Spin^c(4)$ and $\omega_\mC^2 = 1$. This means that the action of $\omega_\mC$ on $\mCl(4)$ has positive and negative eigenspaces. Let $S^\pm$ denote the subbundles of $S$ defined by the positive and negative action of $\omega_\mC$. 

Now, we can work to define the Seiberg-Witten map, first developed in \cite{Seiberg_1994,Seiberg_1994_2}. Let $A$ be a 1-form on $X$ valued in $\fru(1)= i\mR$. Let $$\gamma\colon TX\to \Hom(S^+,S^-)\subset \End(S)$$ denote the induced Clifford multiplication which works as one would expect fiberwise: a tangent vector Clifford multiplies a positive spinor, transforming the spinor into a negative spinor. $\nabla_A$ is defined to be a unitary connection 
such that, for any $v\in\Gamma(TX)$ and $s\in \Gamma(S)$ 
$$\nabla_A(\gamma(v)s) = \gamma(v)\nabla_A(s)+\gamma(\nabla_{\text{LC}}v)s$$ where $\nabla_{\text{LC}}$ is the Levi-Civita connection for $(X,\frs,g)$. 
Let $F_A = \nabla_A\circ \nabla_A$ denote the curvature of the spin$^c$ connection $\nabla_A$, and $F_A^+$ denote the self-dual part of the curvature. Let $(-)_0$ denote the traceless part of an operator. We take the traceless part, since, on a vector bundle $E$, $$\End(E)\cong \eps \oplus \End_0(E).$$ This is because the identity map on fibers defines a nowhere vanishing global section of $\End(E)$, allowing us to peel off a copy of $\eps$. Let $\dirac_A$ denote the Dirac operator, twisted by $A$, given by 
$$\Gamma(S)\xr{\nabla_A}\Gamma(T^\ast X\otimes S)\xr{\sim} \Gamma(TX\otimes S)\xr{\gamma}\Gamma(S).$$ The Dirac operator can be thought of as a ``square root of the Laplacian'' which works for spin manifolds. In the arbitrary case, it may have some extra curvature terms when ``squared,'' but this difference is explained by the Lichnerowicz-Weitzenb\"{o}ck formula, discovered in \cite{Lichnerowicz1963}, which can be read about in course notes \cite{shintaro-notes-manolescu-283A-4-manifolds} and the book \cite{spingeometrylawsonmichelsohn}.  

Let $\sigma$ denote a section of the positive spinor bundle. This gives us enough to write the Seiberg-Witten map down:
\begin{center}
  \begin{tikzcd}[ampersand replacement = \&]
    \SW\colon\Gamma(S^+)\oplus \Omega^1(X;i\mR)\ar[r]\& \Gamma(S^-)\oplus \Omega^2_+(X;i\mR) \\ [-20pt]
    (\sigma,A)\ar[r,mapsto] \& (\dirac_A \sigma, F_A^+-\gamma\inv(\sigma\otimes \sigma^\ast)_0).
  \end{tikzcd}
\end{center}

The \textbf{Seiberg-Witten equations} are what we get when we try to find the zeros of this map, $$\dirac_A\sigma = 0,\qquad F_A^+ = \gamma\inv(\sigma\otimes \sigma^\ast)_0.$$ Before taking completions, this situation is analytically difficult. Therefore, we take the $L^2_k$-Sobolev completion of our source and the $L^2_{k-1}$-Sobolev completion of our target. This leads to more manageable analysis and nice properties of the zeros of this new map $\SWhat$ between the Sobolev completions. 

In our situation, $X$ is a spin manifold, so $\SWhat$ is $\Pintwo$-equivariant. 
We now set some notation for the source and target spaces. Let
\begin{align*}
    \calW^+ &:= \parenths{\Gamma(S^+)\oplus \Omega^1(X;i\mR)}^{\hat{~}}_{L^2_k}, & \calW^- &:=\parenths{\Gamma(S^-)\oplus \Omega^2_+(X;i\mR)}^{\hat{~}}_{L^2_{k-1}}
\end{align*}
These $\Pintwo$-Hilbert spaces can be split as 
\begin{align*}
\calV^+& := \parenths{\Gamma(S^+)}^{\hat{~}}_{L^2_k} & \calV^- &:= \parenths{\Gamma(S^-)}^{\hat{~}}_{L^2_{k-1}}\\
\calU^+& := \parenths{\Omega^1(X;i\mR)}^{\hat{~}}_{L^2_k} & \calU^- &:=\parenths{\Omega^2_+(X;i\mR)}^{\hat{~}}_{L^2_{k-1}}.
\end{align*} 
$\calV^\pm$ are quaternionic Hilbert spaces, endowed with a natural action by $\Pintwo\subset \text{Cl}(2)\cong \mH$ through left and right multiplication. $\calU^\pm$ are real Hilbert spaces, with left and right $\Pintwo$ actions given by $e^{i\theta}\cdot u = u$ and $j\cdot u = -u$. 

The Seiberg-Witten map is a compact Fredholm operator. That is, $\SWhat$ can be decomposed as $\SWhat = \ell + c$ where $\ell$ is Fredholm and $c$ is compact.  Fredholm means $\ell$ has finite-dimensional kernel and cokernel, and compact means the image of any bounded subset under $c$ is precompact. 

\subsubsectionlabel{Bauer-Furuta Invariant}
In their paper, \cite{Bauer_Furuta_SCR} and, subsequently, \cite{bauer_stable_cohomotopy_ii}, Bauer and Furuta use properties of this map, $\SWhat$, to define a stable map between spheres. We outline this process here. 

First, we set a complete $\Pintwo$-equivariant Grothendieck universe $\calU$ which contains, as a subspace, the concrete $\Pintwo$-representation $$\calU':=\bigoplus_\infty \mR\oplus \bigoplus_\infty \mtR\oplus \bigoplus_\infty \mH.$$ We denote $\Pintwo$-equivariant homotopy classes of maps between based, left $\Pintwo$-spaces $X,Y$ as $[X,Y]^{\Pintwo}$. We denote the stable homotopy classes of maps between based, left $\Pintwo$-spaces $X,Y$ as $\sets{X,Y}^{\Pintwo}$. We produce this object as the colimit over finite-dimensional subspaces of $\calU$, $$\sets{X,Y}^{\Pintwo} := \underset{\underset{fin. dim.}{V\subset\calU}}\colim~ [S^V\wedge X,S^V\wedge Y]^{\Pintwo}.$$ 
In our case, the incomplete universe $\calU'$ will suffice because our $\Pintwo$-CW complexes have isotropy groups $\Pintwo, S^1$, or $e$ \cite[pg. 5]{JLin20}. 
In the case where $X = S^U$ and $Y = S^W$ for some finite-dimensional $\Pintwo$-representations $U,W$, we use any of the equivalent notations to denote stable homotopy classes of maps, $$\pi_{U-W}^{\Pintwo}(S^0) = \pi_{U}^{\Pintwo}(S^W) = \sets{S^{U},S^{W}}^{\Pintwo}.$$ 

With notation settled, we can begin to define the Bauer-Furuta invariant. Let $W^-\subset \calW^-$ be a finite-dimensional subspace transverse to the image of $\ell$. Set $W^+:=\ell\inv(W^-)$. Be sure to choose $W^-$ such that $W^-\cong x\mH\oplus y\mtR$ for some natural numbers $x,y$. With these finite-dimensional vector spaces, we can define the approximated Seiberg-Witten map, which we denote $$\SWapp \equiv \ell+\proj_2 c\colon W^+\to W^-.$$ We also set $V^\pm$ to be those subspaces of $W^\pm$ which are quaternionic and  $U^\pm$ to be those subspaces isomorphic to finitely many summands of $\mtR$. 

Now, by compactness of $c$ and linearity of $f$, there exists $R,\eps>0$ such that $\SWhat\inv(B_\eps(0))\subset B_R(0)$. One can even ensure that $\SWhat(\dee B_{R+1})\subset \ol{W^-\setminus B_\eps(0)}$. With this in hand, we're able to define a map $\SWappsph\colon B_{R+1}(0)/\dee B_{R+1}(0) \to W^-/\ol{W^-\setminus B_\eps(0)}.$ Recognizing the source and target as the one-point compactifications of spheres, we see the map $$\SWappsph\colon S^{W^+}\to S^{W^-}.$$ 

One can show that the choices made earlier, such as the subspace $W^-$ and the lengths $R,\eps$, are independent of the $\Pintwo$-stable homotopy class of map realized by $\SWappsph$. Therefore, the Bauer-Furuta invariant is a well-defined invariant $$\BF^{\Pintwo}(X,\frs):=[\SWappsph]\in \pi^{\Pintwo}_{{W^+}}(S^{W^-}).$$ It is a theorem of Furuta \cite[Theorem 4.2]{furuta11over8}, 
that if a simply-connected four-manifold has intersection form $2pE_8\oplus q\Hmatrix$, then $$\BF^{\Pintwo}(X,\frs)\in \pi^{\Pintwo}_{p\mH}(S^{q\mtR})$$ and fits into a diagram of $\Pintwo$-spaces
\begin{center}
    \begin{tikzcd}[ampersand replacement = \&]
      \& S^0\ar[dl,"a^p_\mH"']\ar[dr,"a^q_{\mtR}"]\\
      S^{p\mH}\ar[rr,"{\BF^{\Pintwo}(X,\frs)}"]\&\& S^{q\mtR}
    \end{tikzcd},
\end{center} where $a^p_V$ denotes the inclusion of $0,\infty$ into $S^{pV}$. A map which satisfies the diagram property is known as a Furuta-Mahowald class, and the study of these maps fruitfully led to major progress on Matsumoto's 11/8-conjecture in \cite{Hopkins_Lin_Shi_Xu} as stated in \cite{furuta11over8}. 

\subsubsectionlabel{Families Bauer-Furuta Invariant}
The Bauer-Furuta invariant also has a family version, defined by Szymik and Xu \cite{szymik2020characteristic,Xu2004}. In this paper, we focus on the specific context where our family is over $S^1$. Here, we will define the construction we use, but note that the general situation can be found in the aforementioned resources. 

Let $(X,\frs)$ be a simply-connected four-manifold with spin structure $\frs$. Let $(E,\frs_E)$ be a smooth fiber bundle $\pi:E\to S^1$ with fiber $X$ and $\frs_E$ a spin structure on the vertical tangent bundle of $E$, $T^vE :=\ker(D\pi)$, which restricts to $\frs$ on each fiber. 

Our example comes from a diffeomorphism $f:X\to X$ and taking the mapping torus of $f$, $$T_f:=\frac{X\by [0,1]}{(x,0)\sim (f(x),1)}.$$ $f$ induces a map $df:\Fr(X)\to \Fr(X)$, which again has two lifts $\tilde{df}, \tilde{df}^\tau:\frs\to \frs$ between the spin structure on $X$. $\tilde{df}$ is just the product lift and $\tilde{df}^\tau$ is the twisted lift. This allows us to create the family spin structures, $\tilde{\frs}_f, \tilde{\frs}^\tau_f$ defined as 
\begin{align*}
  \tilde{\frs}_f:=T_{\tilde{df}} &= \frac{\frs\by [0,1]}{(p,0)\sim (\tilde{df}(p),1)},&   \tilde{\frs}^\tau_f:=T_{\tilde{df}^\tau} &= \frac{\frs\by [0,1]}{(p,0)\sim (\tilde{df}^\tau(p),1)}.
\end{align*} 
We will define Hilbert bundles associated to $\tilde{df}$, and the $\tilde{df}^\tau$ case is similar. Let $$\tilde{\calW}^\pm:=\frac{\calW^\pm\by [0,1]}{(w,0)\sim (\tilde{df}_\ast(w),1)}.$$

\begin{example}\label{example:key-families-over-circle-identity}
  We include the example of the identity diffeomorphism for concreteness and to come back to later in the argument. In this case 
  \begin{align*}
      \tilde{\frs}_{\uno}:=T_{\tilde{d\uno}} &= \frac{\frs\by [0,1]}{(p,0)\sim (p,1)},&   \tilde{\frs}^\tau_{\uno}:=T_{\tilde{d\uno}^\tau} &= \frac{\frs\by [0,1]}{(p,0)\sim (-1\cdot p,1)}.
  \end{align*} 
  This also allows us to describe how the families Seiberg-Witten map should go for each. We first construct the associated Hilbert bundles for each spin structure. For the family spin structures $\tilde{\frs}_{\uno},\tilde{\frs}^\tau_\uno$, we set 
  \begin{align*}
    \tilde{\calW}^\pm&:= \frac{\calW^\pm\by [0,1]}{(w,0)\sim (w,1)}& \tilde{\calW}^\pm_{\tau}&:=\frac{\calW^\pm\by [0,1]}{(w,0)\sim (\tilde{d\uno}^\tau_\ast(w),1)}\\
    &= \frac{\calU^\pm\oplus \calV^\pm\by [0,1]}{(u,v,0)\sim(u,v,1)}, & &=\frac{\calU^\pm\oplus \calV^\pm\by [0,1]}{(u,v,0)\sim(u,-v,1)}.\\
    &=\ul{\calU}^\pm\oplus \ul{\calV}^\pm & &= \ul{\calU}^\pm \oplus \frac{\calV^\pm\by [0,1]}{(v,0)\sim (-v,0)}.
  \end{align*} 
\end{example}

\lineyspace

With this in hand, we can define a compatible families Seiberg-Witten map, $$\FSWhat:\tilde{\calW}^+\to \tilde{\calW}^-,$$ which restricts to the Seiberg-Witten map on fibers. $\FSWhat = \tilde{\ell}+\tilde{c}$ where $\tilde{\ell}$ is a fiberwise Fredholm map and $\tilde{c}$ is a fiberwise compact operator. By Kuiper's theorem \cite{Kuiper1965TheHT}, $$\tilde{\calW}^-\cong S^1\by\calW^-.$$ Applying a similar process as the one we used above, choose a subbundle $S^1\by W^-\subset S^1\by\calW^-$ such that $S^1\by W^-+\tilde{\ell}(\tilde{\calW}^+) = S^1\by\calW^-$. This can be done since $S^1$ is a compact base and we can trivialize along neighborhoods and take a large enough $W^-$ to cover all fibers. Set $\tilde{W}^+:= \tilde{\ell}\inv(W^-)$. Now, we can choose sub-disk bundles $D(\tilde{W}^+), D(S^1\by W^-)$. Since $\FSWhat$ sends 0 sections to 0 sections and $\infty$-sections to $\infty$-sections, we can further restrict this to the Thom spaces $$\FSWhatappthom:\Th(\tilde{W}^+)\to \Th(S^1\by W^-) = S^1_+\wedge S^{W^-}\to S^{W^-}.$$ This invariant is a parametrized stable homotopy invariant, i.e., it is in the stable homotopy group $$\FSWhat^\Pintwo(E,\frs_E)\in \sets{\Th(\tilde{W}^+),S^{W^-}}^\Pintwo.$$

The main diffeomorphism in the arguments of this paper comes from the identity diffeomorphism on $\K3\#\K3\setminus B^4$. From arguments in \cite{KM20,JLin20} we can canonically, up to homotopy, trivialize $\tilde{W}^+$ as $S^1\by W^+$, so the map can be found in 
$$\FSWhatappthom\in\sets{S^1_+\wedge S^{W^+},S^{W^-}}^\Pintwo.$$

\subsubsectionlabel{Useful Relationships in Equivariant Homotopy}
The (families) Bauer-Furuta invariant has multiple incarnations depending on the equivariance of the Seiberg-Witten map you are choosing to involve. These are related through the short exact sequence $$\sets{e}\to S^1\to \Pintwo\to C_2\to \sets{e}.$$ There are restriction maps which take the (families) Bauer-Furuta map between the different $\Pintwo, S^1,$ and non-equivariant homotopy classes of maps. In general, for left-$\Pintwo$ spaces $X,Y$, there are maps between their different equivariant stable homotopy classes of maps, encapsulated by the restriction maps $$\sets{X,Y}^{\Pintwo}\xr{\text{Res}_{S^1}^{\Pintwo}}\sets{X,Y}^{S^1}\xr{\text{Res}_{\sets{e}}^{S^1}}\sets{X,Y}^{\sets{e}}.$$

\subsubsectionlabel{Properties of the Bauer-Furuta invariant}
The Bauer-Furuta invariant is a refinement of the Seiberg-Witten invariant, as there is a homomorphism, the characteristic homomorphism $$t:\sets{S^{a\mH+b\mR},S^{c\mtR}}^\Pintwo\to \mZ,$$ which takes the Bauer-Furuta invariant of $(X,\frs)$ to the Seiberg-Witten invariant, an integer. The explicit description of this was written in \cite{Bauer_Furuta_SCR}, and exposition with similar notation to this paper can be found in \cite{JLin20,BirgitSchmidt2003Spin4A}.

In \cite[\S38.2]{KM07}, Kronheimer and Mrowka provide a formula for $\SW(\K3,\frs)$ which equals 1. Since there is the characteristic homomorphism $t$, this means that, from Furuta's theorem, $$\BF^{\sets{1}}(\K3,\frs)\in \pi_1^s=\text{Res}_{\sets{1}}^\Pintwo\sets{S^{\mH},S^{3\mtR}}^\Pintwo$$ is nonzero, meaning that it must be $\eta$. 

There is a vanishing theorem for the Seiberg-Witten invariant, stating: if $X = X_0\#X_1$ with $b_2^+(X_i)>0$, then $\SW(X) =0$, found at \cite[2.4.6]{Gompf-Stip-Kirby-Book}. In a follow-up paper, \cite{bauer_stable_cohomotopy_ii}, Bauer proves that if $Y = X_0\#X_1$ with spin structure $\frs_Y$ (some compatible connected sum of the spin structures $\frs_0,\frs_1$ of $X_0, X_1$) then $$\BF^{\Pintwo}(Y,\frs_Y) = \BF^{\Pintwo}(X_0,\frs_0)\wedge \BF^{\Pintwo}(X_1,\frs_1),$$ the smash product of the two maps. This shows us that$$\BF^{\sets{1}}\parenths{(\K3)^{\# n},\frs} = \eta^n.$$ Note that $\pi_4^s = 0$, so $\eta^n$ is zero for $n\geq 4$. Of course, we can tell all of these connected sums of $\K3$ surfaces apart by Betti number calculations, but this shows that the Bauer-Furuta invariant is separating these surfaces finer than just the integral Seiberg-Witten invariant.

Now, we show the equivariant Bauer-Furuta invariant is a finer sieve than the non-equivariant version. In her thesis \cite[Cor. 4.4]{BirgitSchmidt2003Spin4A}, Schmidt does find that for $\Pintwo$-equivariant stable Hopf map $$\eta_\Pintwo:S^\mH\to S^{3\mtR},$$ any power of it, $\eta^n_{\Pintwo}$, remains nontrivial in the $\sfR\sfO(\Pintwo)$ graded stable homotopy groups $\pi^{\Pintwo}_\star(S^0)$ by passing to geometric fixed points under the $\Pintwo$ action. She also finds  that $$\BF^{\Pintwo}(\K3,\frs) = \eta_\Pintwo+(\text{other terms})$$ in \cite[Thm. 7.1]{BirgitSchmidt2003Spin4A}. So, using $\Pintwo$-equivariant homotopy theory, the different connected sums of $\K3$s each have invariants which are different elements of the stable equivariant homotopy groups. Still, Betti numbers are able to distinguish these, of course.

To be clear, the Bauer-Furuta invariant is indeed a refinement on the Seiberg-Witten invariant. Fintushel and Stern show that there are infinitely many inequivalent smooth structures on a homotopy $\K3$ in their paper \cite[Cor. 1.4]{FintushelStern1998}. If one takes a connected sum of these $\K3$ surfaces with the standard one, the Seiberg-Witten invariant vanishes. On the other hand, they can still be distinguished by the non-equivariant Bauer-Furuta invariant for various spin$^c$ structures. The $\Pintwo$-equivariant Bauer-Furuta invariant is also indeed a further refinement. For example, the boundary Dehn twist from the punctured $\K3\#S^2 x S^2$ can be detected by $\Pintwo$-equivariant families Bauer-Furuta invariant, but not the non-equivariant.

\subsectionlabel{Pin(2)-bundles and Actions}

In this section, we collect some useful actions of $\Pintwo$ on some spaces, some related subgroups, and quotients on a variety of spaces. $\Pintwo$ fits in the short exact sequence 
\begin{equation}
  1\to S^1\to \Pintwo\to C_2\to 1.\label{eqn:ses-pintwo}
\end{equation} 
$\Pintwo$ is also a subgroup of $\Sp(1)\cong \SU(2)\cong S(\mH)$. The contractible space $S^\infty\subset \mH^{\oplus \infty}$ is both a left and right $\Pintwo$-space by coordinatewise multiplication. Some quotients of $S^\infty$ by several subgroups are collected below 
\begin{align*}
  _{S^1\setminus} S^\infty&=\mCP^\infty \text{ is a $BS^1$,} & _{\Pintwo\setminus} S^\infty&=_{C_2\setminus}\mtCP^{\infty}\text{ is a $B\Pintwo $,} & _{\Sp(1)\setminus}S^{\infty}&=\mHP^\infty \text{ is a $B\Sp(1)$.}
\end{align*}
$C_2$ acts on $\mCP^\infty$ as the quotient of the action of $\Pintwo$ on $S^\infty$, specifically by $$[j]\cdot[\alpha_0:\beta_0:\alpha_1:\beta_1:\cdots] = [-\ol{\beta}_0:\ol{\alpha}_0:-\ol{\beta}_1:\ol{\alpha}_1:\cdots].$$

The map $S^1\into \Sp(1)$ induces a map $BS^1\to B\Sp(1)$, witnessed as the map \begin{center}
  \begin{tikzcd}[ampersand replacement = \&]
    \mtCP^\infty\ar[r]\& \mHP^\infty \\ [-20pt]
    [\alpha_0:\beta_0:\alpha_1:\beta_1:\cdots]\ar[r,mapsto] \& {[\alpha_0+\beta_0\mathbf{j}:\alpha_1+\beta_1\mathbf{j}:\cdots]}
  \end{tikzcd}
\end{center} which has $S^2$ as its fiber. This is because of the quotient of $S^3/S^1 = S^2$ in the Hopf fibration $S^1\to S^3\to S^2.$ 
Furthermore, the inclusion $\Pintwo \into \Sp(1)$ induces $p:B\Pintwo\to \mHP^\infty$ which has fiber $\mRP^2$ since the quotient is by an additional $C_2$-action coming from $j$. Note that $$p\inv(\mHP^n) = B\Pintwo^{(4n+2)},$$ the $(4n+2)$-skeleton of $B\Pintwo$, and is the total space of an $\mRP^2$-bundle over $\mHP^n$.

\begin{remark}\label{constr:s2-bundles}
  Let $F$ be a topological space with a left $\Pintwo$ action and a right $S^1$-action such that for any $f\in F$, $-1\cdot f = f\cdot -1$. Then there exists a map 
  \begin{center}
    \begin{tikzcd}[ampersand replacement = \&]
      \ctheta:S^1\by F\ar[r]\& F \\ [-20pt]
      (z,f)\ar[r,mapsto] \& \sqrt{z}\cdot f \cdot \ol{\sqrt{z}}.
    \end{tikzcd}
  \end{center}
  Let $\cthetabar$ denote the map where the conjugation is flipped, i.e., $\ol{\sqrt{z}}\cdot f \cdot \sqrt{z}$. Though the square root of a complex number is not well-defined, this map is well-defined. Two square roots of a complex number $z$ differ by a factor of $-1$. This means that, by construction, $$(-\sqrt{z})\cdot f \cdot (-\ol{\sqrt{z}}) = \sqrt{z}\cdot f\cdot\ol{\sqrt{z}}.$$ 

  We define a new space $E_F$ via a clutching construction. $$E_F:=D^2\by F\cup_{S^1\by F}D^2\by F$$ where the attaching map is given as 
  \begin{center}
    \begin{tikzcd}[ampersand replacement = \&]
      S^1\by F\ar[r]\& S^1\by F \\ [-20pt]
      (z,f)\ar[r,mapsto] \& (z,\ctheta(z,f)).
    \end{tikzcd}
  \end{center}
  This creates an $F$-bundle over $S^2$. Note that $\ctheta$ is not $\Pintwo$-equivariant, but, if we quotient by the left $S^1$-action, $_{S^1\setminus}\ctheta$ is $C_2$-equivariant. When $F$ is a representation sphere $S^V$, we write $E_V$ for $E_{S^V}$ for brevity. 
\end{remark}\lineyspace

\subsectionlabel{Some Complex Representation Theory}
The main groups at play in our paper come from the short exact sequence in Equation (\ref{eqn:ses-pintwo}). It is good to have in mind the complex representation rings associated to each of these and certain transfer and restriction maps between them. A collection of results we define in this section can be found in \cite{manolescu2014intersection}, and these facts come from the books and papers \cite{atiyah2018k,atiyah1968bott,BirgitSchmidt2003Spin4A,Segal}

$\sfR(G)$ is the Grothendieck ring of all complex representations of $G$ where we allow $\ominus$ to work as an inverse to $\oplus$ and  $\otimes$ is our product. For our purposes, we need to know the complex representation rings of the groups $S^1, \Pintwo, C_2$, and $\sets{e}$. The final group in this list has an easy representation ring of $\sfR(\sets{e}) = \mZ$.  

An important operation in representation rings is the $\lambda_k$ operation, defined on a representation $V$ as 
\begin{equation}
  \lambda_{k}(V) = \sum_{n=0}^\infty k^n V^{\wedge n}.\label{eqn:lambda-k}
\end{equation} 
We particularly care about the case $k = -1$,  
\begin{equation}
  \lambda_{-1}(V) = \sum_{n=0}^\infty (-1)^n V^{\wedge n}.\label{eqn:lambda-neg-one}
\end{equation} 
This is because, in equivariant Bott periodicity, the isomorphism of $\sfK$-groups \begin{center}
  \begin{tikzcd}[ampersand replacement = \&]
    \tilde{\sfK}_G(\Sigma^V X)\ar[r,"\simeq"]\& \tilde{\sfK}_G(X) \\ [-20pt]
    [E]-[F]\ar[r,mapsto] \& \lambda_{-1}(V)([E]-[F]).
  \end{tikzcd}
\end{center} We call the $\lambda_{-1}(V)$ the \textbf{$\sfK$-theoretic Euler class of the representation $V$}. This $\lambda_{-1}(V)$ is sometimes written as $a_V$. If we have $n$ direct sums of $V$, we can write $a_{nV} = a^n_V$. Since we work in $\sfK$-theory which deals with complex vector bundles, the $\sfK$-theoretic Euler classes must work with complex vector spaces.  There are $\sfK\sfO$-theoretic Euler classes as well which work similarly. 

Consider the inclusion $0\into V$. If we take the one-point compactification of $V$, viewing the one point as $\infty$, and separately we add $\infty$ to the set $\sets{0}$, we can define a map $i:S^0\into S^V$. In $\sfK_G$-theory, this is the map $$\sfK_G(S^V)\xr{i^\ast}\sfK_G(S^0)$$ which induces multiplication by $\lambda_{-1}(V)$ aka $a_V$. 

This $\lambda_{-1}$ map also commutes with restriction maps. We will write the representation rings based on concrete representations as generators and then write the rings using the $\lambda_{-1}$ versions of the generators. The latter will be more convenient for us in the $\sfK$-theory calculations of \S\ref{subsec:Equivariant K-theory}. 

\begin{description}
  \item[($S^1$)] $S^1$ has representations $\mC$ with trivial $S^1$ action, $\theta$ which is $\mC$  with the usual multiplication, and $\ol{\theta}$ which is $\mC$ with conjugate multiplication. The representation ring is then $$\sfR(S^1) \cong \frac{\mZ[\theta,\ol\theta]}{(\theta\ol\theta-1)}.$$
  
  Let $a := \lambda_{-1}(\theta) = 1-\theta$ and $b := \lambda_{-1}(\ol{\theta}) = 1-\ol\theta$. This gives an alternative presentation, $$\ds \sfR(S^1) \cong \frac{\mZ[a,b]}{(ab-a-b)}.$$
  \item[($C_2$)] $C_2$ has representations $\mC$ with trivial $C_2$ action, and $\sigma$, the one-dimensional representation with the action of multiplying by $-1$. The representation ring is then $$\sfR(C_2) \cong \frac{\mZ[\sigma]}{(\sigma^2-1)}.$$
  
  Let $\tilde{w} := \lambda_{-1}(\sigma) = 1-\sigma$. This gives an alternative presentation, $$\ds \sfR(C_2) \cong \frac{\mZ[\tilde{w}]}{(\tilde{w}^2-2\tilde{w})}.$$
  \item[($\Pintwo$)] $\Pintwo$ has representations $\mC$ with trivial left $\Pintwo$ action, $\mtC$ with action given by $e^{i\theta}\cdot z = z$ and $j\cdot z = -z$, and $\mH$ with left action given by left multiplication. The representation ring is then $$\sfR(\Pintwo) \cong \frac{\mZ[[\mtC],[\mH]]}{([\mtC]^2-1,[\mtC][\mH]-[\mH])}.$$
  
  Let $w := \lambda_{-1}(\mtC) = 1-\mtC$ and $z := \lambda_{-1}(\mH) = 2-\mH$. This gives an alternative presentation, $$\ds \sfR(\Pintwo) \cong \frac{\mZ[w,z]}{(w^2-2w,zw-2w)}.$$ Note that in this paper $w,z$ correspond to the $\sfK$-theoretic Euler classes $a_{\mtC},a_{\mH}$ respectively. 
\end{description}

We include some words about these maps under the restriction homomorphism. Restricting to the trivial representation, each of the lower-case Roman-letter generators, $a,b,z,w$ has virtual dimension 0. For the restriction of $\Pintwo$ to $S^1$, we explicitly write the image of the generators here: 
\begin{align*}
  \Res_{S^1}^\Pintwo(w)&=0, & \Res_{S^1}^{\Pintwo}(z)&=a+b.
\end{align*}

This concludes the prerequisite material. We are now prepared to move on to the problem.

\sectionlabel{Application}

\subsectionlabel{Spaces and Facts}
Let $X:=\K3\#\K3$ and let $X^\circ:=(X\setminus B^4,\dee X)$ where the 4-ball is removed away from the neck. Let $\frs$ denote the unique spin structure on $X$. Let $\uno$ denote the identity map on $X^\circ$, and let $\delta^\circ$ denote the boundary Dehn twist on $X^\circ$. One can extend the Dehn twist to all of $X$ by making the diffeomorphism the identity on the removed four-ball, and we call this $\delta$.

Let $T_\uno,T_\delta$ denote the mapping tori of the identity map and the Dehn twist. 

In both cases, $\sfH^1(T_f;\mZ/2) = \mZ/2$, so both $T_\uno, T_\delta$ have two family spin structures. The construction of their family spin structures was laid out in \S\ref{subsubsec:Families Bauer-Furuta Invariant} and for the identity specifically in Example \ref{example:key-families-over-circle-identity}.

Since, in a collar neighborhood $N'$ of $\dee X^\circ$, $\delta^\circ$ is the identity, we know there is an identification $$\psi: T_{\delta}|_{N'\by S^1}\xr{\sim}T_\uno|_{N'\by S^1},$$ wherein the source is a subset of $T_\delta$ and the target is a subset of $T_\uno$. For $N'\by S^1\subset T_\uno$, we let $\tilde{\frn},\tilde{\frn}^\tau$ denote $\tilde{\frs}|_{N'\by S^1},\tilde{\frs}^\tau|_{N'\by S^1}$, respectively. We then denote the two spin families over $T_\delta$ by $\tilde{\frs'}$ and $\tilde{\frs'}^\tau$ where we enforce that $\tilde{\frs'}$ corresponds to the bundle where $\tilde{\frs'}|_{N'\by S^1} = \psi^\ast\tilde{\frn}$ and $\tilde{\frs'}^\tau|_{N'\by S^1} = \psi^\ast\tilde{\frn}^\tau$.

\begin{proposition}\label{prop:n-bundle-prop}
  There is a map $\phi\colon T_\uno\to T_\delta$ which induces the isomorphisms $$(T_\uno,\tilde\frs)\cong (T_\delta,\tilde{\frs'}^\tau)\text{ and }(T_\uno,\tilde\frs^\tau)\cong (T_\delta,\tilde{\frs'}).$$
\end{proposition}
\begin{proof}
  Let 
  \begin{align*}
    T_\uno^\circ&=\frac{X^\circ\by [0,1]}{(x,0)\sim (x,1)}\text{ and}& T_\delta^\circ &=\frac{X^\circ\by [0,1]}{(x,0)\sim (\delta(x),1)},
  \end{align*}
  and take note of the fact that $N'\by S^1$ is a subspace of both of these. let us call $N' =S^3\by [0,1]$ and $N'' = S^3\by [1/2,1]$. Define a monotonically increasing smooth function $r:I\to I$ such that $r(t)$ is 0 near 0 and is 1 for $t\geq \frac{1}{2}$. We define a map \begin{center}
    \begin{tikzcd}[ampersand replacement = \&]
      \phi_{N'}:T_\uno^\circ |_{N'\by I/\sim}\ar[r]\&  T_\delta|_{N'\by I/\sim} \\ [-20pt]
      {[((x,y),t,s)]}\ar[r,mapsto] \& {[((e^{ir(t)s}x,y),t,s)]}
    \end{tikzcd}.
  \end{center}
  We can extend $\phi_{N'}$ to $\phi:T^\circ_\uno\to T^\circ_\delta$ via the identity. Note that $\phi|_{s=0}$ is the identity, and $\phi|_{s=1}$ is (smoothly homotopic to) the Dehn twist.

  Now, consider the fixed points of the map $\phi_{N'}$ given by  $[((0,1),1,s)]\in T_\uno^\circ$ and let $z$ denote the triple $((0,1),1)$. When we apply $\phi(-,s)_\ast$ to $T_z N''$ for each $s\in I/\dee I$, we see that we get a nontrivial loop in $\pi_1(\SO(4))$. Therefore, $\phi^\ast(\tilde{\frn}^\tau) = \tilde{\frn}$, and extending to the whole of $T_\uno$ and $T_\delta$ by the identity yields the stated isomorphism in the proposition. 
\end{proof}

This yields the following corollary, whose contrapositive determines whether or not $\delta\simeq \uno$.
\begin{corollary}
  If $\delta^\circ\simeq \uno$ in $\Diff_\dee(X^\circ)$, then $$\FBF^\Pintwo(T_\uno,\tilde{\frs}) = \FBF^{\Pintwo}(T_\uno,\tilde{\frs}^\tau).$$
\end{corollary}
\begin{proof}
  Suppose that $\delta^\circ\simeq \uno$ in $\Diff_\dee(X^\circ)$, then $(T_\uno, \tilde{\frs})\cong (T_\delta,\tilde{\frs'})$ and $(T_\uno, \tilde{\frs}^\tau)\cong (T_\delta,\tilde{\frs'}^\tau)$. This and \ref{prop:n-bundle-prop} yield isomorphisms $$(T_\uno,\tilde{\frs})\cong (T_\delta,\tilde{\frs'}^\tau)\cong (T_\uno,\tilde{\frs}^\tau).$$ This implies $$\FBF^{\Pintwo}(T_\uno,\tilde{\frs}) = \FBF^{\Pintwo}(T_\uno,\tilde{\frs}^\tau).$$ 
\end{proof}

The rest of the paper uses this corollary to determine that the boundary Dehn twist cannot be isotopic to $\uno$ on $X$. 

\sectionlabel{Calculations}
\subsectionlabel{Explicit Description of the two FBF maps on the mapping torus}

Let $\FBF(\tilde{\frs}),\FBF(\tilde{\frs}^\tau)$ denote $\FBF^{\Pintwo}(T_\uno,\tilde\frs), \FBF^{\Pintwo}(T_\uno,\tilde\frs^\tau)$, respectively. Since, by hypothesis $[\FBF(\tilde{\frs})] = [\FBF(\tilde{\frs}^\tau)]$ in the set $\sets{S^1_+\wedge S^{2\mH},S^{6\mtR}}^\Pintwo$, this means, for $V$ a large enough $\Pintwo$ representation in our universe $\calU$, there is always a homotopy $$H_V \colon S^1_+\wedge S^{2\mH+V}\wedge I_+\to S^{6\mtR+V}$$ where $H_V|_{0} = \FBF(\tilde{\frs})_V$ and $H_V|_{1} = \FBF(\tilde{\frs}^\tau)_V$. 

First, let us calculate $\FBF(\tilde{\frs}), \FBF(\tilde{\frs}^\tau)$ explicitly on the space level. Let $\calW^+$ denote the source and let $\calW^-$ denote the target of the Seiberg-Witten map on $(X,\frs)$. Let $W^\pm\subset \calW^\pm$ be subspaces which satisfy the requirements to define the Bauer-Furuta map. Recall that $$W^\pm \cong U^\pm \oplus V^\pm,$$ with $U^\pm$ isomorphic to a finite direct sum of the representation $\mtR$ and $V^\pm$ isomorphic to a finite direct sum of the representation $\mH$. 

For each $\theta\in [0,\pi]$, define the bundles over $S^1$ as $$_\theta W^\pm:= \frac{W^\pm\by [0,1]}{(w,0)\sim (e^{i\theta}w,1)} \cong \ul{U}^\pm \oplus \frac{V^{\pm}\by [0,1]}{(v,0)\sim (e^{i\theta}v,1)}.$$ We put the $\theta$ on the left to denote left-multiplication. Since $\SWapp$ is $\Pintwo$-equivariant on the left, we can easily define the map 
\begin{center}
  \begin{tikzcd}[ampersand replacement = \&]
    _\theta\SW: \,_\theta W^+\ar[r]\& _\theta W^- \\ [-20pt]
    {[(w,t)]}\ar[r,mapsto] \& {[(\SW(w),t)].}
  \end{tikzcd}
\end{center}

Consider also the collection of bundle maps
\begin{center}
  \begin{tikzcd}[ampersand replacement = \&]
    c_\theta:\,_\theta W^\pm\ar[r]\& _\theta W^\pm \\ [-20pt]
    {[(w,t)]}\ar[r,mapsto] \& {[(e^{-i\pi t}we^{i\pi t},t)]}.
  \end{tikzcd}
\end{center}
Let the inverses of these, where we multiply on the left by $e^{i\pi t}$ and on the right by $e^{-i\pi t}$, be denoted $\ol{c}_\theta$. These are bundle isomorphisms, but they are not quaternionic linear because of the left-multiplication, and hence not $\Pintwo$-equivariant.

Note that, given a choice $W^\pm$, $_0\SW$ can be used to define the stable homotopy class of $\FBF(\tilde{\frs})$ and $_\pi\SW$ for the stable homotopy class of $\FBF(\tilde{\frs}^\tau).$

Following a similar argument to \cite[Prop. 4.1]{KM20}, we get the following lemma. 

\begin{lemma}{}{}\label{lem:prod-twisted-fbf-calc}

  $c_0\circ\, _0\SW\circ\, \ol{c}_0$ can also be used to define the $\Pintwo$-equivariant stable homotopy class $\FBF(\tilde{\frs}^\tau)$. 
\end{lemma}

By ``can define,'' we mean that taking fiberwise one point compactifications and stabilizing, we arrive at the correct stable homotopy class. 
\begin{proof} 
  Let $\textbf{W}^\pm := \ds\bigcup_{\theta\in [0,\pi]} \,_\theta W^\pm$. We can define a map over $[0,\pi]$ given by \begin{center}
  \begin{tikzcd}[ampersand replacement = \&]
    \textbf{SW}:\textbf{W}^+\ar[r]\& \textbf{W}^- \\ [-20pt]
    {([(w,t)],\theta)}\ar[r,mapsto] \& {([\SW(w),t],\theta)}.
  \end{tikzcd}
\end{center} This shows us that $_0\SW$ is homotopic to $_\pi\SW$ linearly but not quaternionic-linearly. Since we are working $\Pintwo$-equivariantly, we would like a quaternionic-linear isomorphism between the two. 

We have the isomorphism 
\begin{center}
  \begin{tikzcd}[ampersand replacement = \&]
    q^{\pm}:\,_\pi W^\pm\ar[r]\& _0W^\pm \\ [-20pt]
    [(w,t)]\ar[r,mapsto] \& {[(we^{i\pi t},t)].}
  \end{tikzcd}
\end{center}
Since the multiplication is done on the right, and right multiplication commutes with left multiplication, we can see that $q^\pm$ are quaternionic-linear isomorphisms. We denote the inverses of $q^\pm$ by $\ol{q}^\pm$ and these are the maps given by right multiplication of $e^{-i\pi t}$.

If we want to compare $_0\SW$ to $_\pi \SW$, they must have same source and target. As long as we choose quaternionic-linear isomorphisms between bundles, the specific isomorphism we pick between $_0W^\pm$ and $_\pi W^\pm$ does not matter. This is because $\Sp(n)$ is 2-connected for all $n\geq 1$. One can prove this inductively looking at the long exact sequence in homotopy groups associated to the fibration $\Sp(n-1)\to \Sp(n)\to S^{4n-1}$. 

Choosing the isomorphisms to be $\ol{q}^+$ and $q^-$, we find that the composition $$q^-\circ \,_\pi \SW\circ\, \ol{q}^+:\,_0 W^+\xr{\ol{q}^+}\,_\pi W^+\xr{_\pi \SW}_\pi W^-\xr{q^-}\,_0W^-$$ allows us to witness $\FBF(\tilde{\frs}^\tau)$ as the bundle map $q^-\circ \,_\pi \SW\circ\, \ol{q}^+$  with source and target $_0W^\pm$ instead of as the bundle map $_\pi\SW$ with source and target $_\pi W^\pm$. 

However, notice $$c_0\circ\, _0\SW\circ\, \ol{c}_0 \equiv q^-\circ \,_\pi \SW\circ\, \ol{q}^+$$ elementwise. Therefore, the bundle map  $c_0\circ\, _0\SW\circ\, \ol{c}_0$ can also define the $\Pintwo$-equivariant stable homotopy class $\FBF(\tilde{\frs}^\tau)$. 
\end{proof}

From Corollary \ref{cor:main-cor}, if the two diffeomorphisms are isotopic, then $\FBF(\tilde{\frs})\simeq \FBF(\tilde{\frs}^\tau)$. 

\begin{corollary}\label{cor:conjugation-corollary}
  If $\FBF(\tilde{\frs})$ is $\Pintwo$-equivariantly stably homotopic to $\FBF(\tilde{\frs}^\tau)$, then $$_0\SW\simeq c_0\circ \,_0\SW\circ \,\ol{c}_0.$$
\end{corollary}

In pursuit of a contradiction, we assume for the remainder of the paper that $\FBF(\tilde{\frs})\simeq \FBF(\tilde{\frs}^\tau)$.

\subsectionlabel{A New Hypothetical Bundle and Some Cofiber Sequences}
We can combine Lemma \ref{cor:conjugation-corollary} with Remark \ref{constr:s2-bundles} to define, for each $\Pintwo$-representation of the form $V=a\mH+2b\mtR$, a bundle map $$\calF\BF_{S^2,V} \colon E_{2\mH+V}\to E_{6\mtR+V}.$$ We can ensure that the summands of $\mtR$ are an even number by the stability of the Bauer-Furuta map; we can always smash another copy of $S^{\mtR}$ if we have an odd number. We use $(f)^+$ to denote the fiberwise one-point compactification of a vector bundle map $f:E\to F$. The map is constructed like so. First, notice the following commutative diagram,  
\begin{equation}\label{eqn:const-fbf-s2-family-map}
    \begin{tikzcd}[ampersand replacement = \&,column sep = 1.5em]
        {D^2\by S^{2\mH+V}}\ar[d,"{\uno\by\SWappsph}"]    \&      {S^1\by S^{2\mH+V}}\ar[l,hook']\ar[r,hook,"{\text{in}_0}"]\ar[d,"{(_0\SW)^+}"]     \&     {I\by S^1\by S^{2\mH+V}}\ar[d,"\uno_I\by H_V"]    \&      {S^1\by S^{2\mH+V}}\ar[l,hook',"{\text{in}_1}"']\ar[r,"{c}"]\ar[d,"{(\ol{c}\circ \,_0\SW\circ c)^+}"]     \&      {D^2\by S^{2\mH+V}}\ar[d,"{\uno\by\SWappsph}"]\\
        {D^2\by S^{6\mtR+V}}    \&     S^1\by S^{6\mtR+V}\ar[r,hook,"{\text{in}_0}"]\ar[l,hook']    \&     I\by S^1\by S^{6\mtR+V}    \&     S^1\by S^{6\mtR+V}\ar[r,"c"]\ar[l,hook',"{\text{in}_1}"']    \&      {D^2 \by S^{6\mtR+V}}
    \end{tikzcd}.
\end{equation}
Taking the colimit of the top row yields the space $E_{2\mH+V}$, similar to the  construction in Remark \ref{constr:s2-bundles}. Taking the colimit of the bottom row, we get $E_{6\mtR+V}$. Since the map from the top row to $E_{6\mtR+V}$ is well defined by the universal property of colimits, we also have a map $E_{2\mH+V}\to E_{6\mtR+V}$, and this is our map $$\calF\BF_{S^2,V}:E_{2\mH+V}\to E_{6\mtR+V}.$$

\begin{remark}
  Notice that, if $a\mH = 0$, the bundle $E_{(2b+6)\mtR}$ is isomorphic to the trivial bundle $$S^2\by S^{(2b+6)\mtR}.$$ In this case we can actually project to the $S^{(2b+6)\mtR}$ fiber via the map $$\proj_2\calF\BF_{S^2,2b\mtR}:E_{2\mH+2b\mtR}\to S^2\by S^{(2b+6)\mtR}\to S^{(2b+6)\mtR}.$$
\end{remark}\lineyspace


The clutching map is not $\Pintwo$-equivariant, but, after quotienting out by the left $S^1$-action, we obtain a $C_2$-equivariant map $$_{S^1\setminus}{\calF\BF_{S^2,V}} \colon ~_{S^1\setminus}{(E_{2\mH+V})}\to _{S^1\setminus}{(E_{6\mtR+V})}.$$ Since the Seiberg-Witten map sends both the zero section to the zero section and the infinity section to the infinity section, we can simplify, and deal with the quotient by these sections, \begin{equation}\label{eqn:fbfs2v}_{S^1\setminus}{\calF\BF_{S^2,V}} \colon \frac{{{_{S^1\setminus}(E_{2\mH+V})}}}{\frac{\sets{[\infty]}\by S^2}{\sets{[0]}\by S^2}}\to \frac{{{_{S^1\setminus}(E_{6\mtR+V})}}}{\frac{\sets{[\infty]}\by S^2}{\sets{[0]}\by S^2}}.\end{equation} We recognize the spaces involved as \begin{align}
  \frac{{{_{S^1\setminus}(E_{2\mH+V})}}}{\frac{\sets{[\infty]}\by S^2}{\sets{[0]}\by S^2}}&=\parenths{E_{S(\mtCP^{2a+ 3})}}_+\wedge S^{2b\mtR},& 
  \frac{{{_{S^1\setminus}(E_{6\mtR+V})}}}{\frac{\sets{[\infty]}\by S^2}{\sets{[0]}\by S^2}}&=\parenths{E_{S(\mtCP^{2a- 1})}}_+\wedge S^{(2b+6)\mtR}.\label{eqn:space-realize}
\end{align} Here we set the convention that, in the case that $a = 0$,  $\mtCP^{-1} = \ast$. 

We also recognize that we have the cofiber sequences $$(E_{\mtCP^{2(a\pm1)+1}})_+\to S^0\to \frac{{E_{S(\mtCP^{2(a\pm 1) +1})}}}{\frac{\sets{[\infty]}\by S^2}{\sets{[0]}\by S^2}}.$$ Smashing these sequences with $S^{2b\mtR}$ and $S^{(2b+6)\mtR}$ yields the cofiber sequences 
$$(E_{\mtCP^{2a+ 3}})_+\wedge S^{2b\mtR}\to S^{2b\mtR}\to (E_{S(\mtCP^{2a+3})})_+\wedge S^{2b\mtR}$$ and 
$$(E_{\mtCP^{2a- 1}})_+\wedge S^{(2b+6)\mtR}\to S^{(2b+6)\mtR}\to (E_{S(\mtCP^{2a-1})})_+\wedge S^{(2b+6)\mtR}.$$ Combining (\ref{eqn:fbfs2v}) and (\ref{eqn:space-realize}) with the above cofiber sequences yields the $C_2$-equivariant commutative diagram 
\begin{equation}\label{eqn:main-diagram-cofibers-spaces}
  \begin{tikzcd}[ampersand replacement = \&]
    (E_{\mtCP^{2a+ 3}})_+\wedge S^{2b\mtR}\ar[r,"p_+"]\& S^{2b\mtR}\ar[r,"i_+"]\ar[d,hook]\& (E_{S(\mtCP^{2a+3})})_+\wedge S^{2b\mtR}\ar[d,"{{_{S^1\setminus}\calF\BF_{S^2,V}}}"]\\
    (E_{\mtCP^{2a- 1}})_+\wedge S^{(2b+6)\mtR}\ar[r,"p_-"]\& S^{(2b+6)\mtR}\ar[r,"i_-"]\& (E_{S(\mtCP^{2a-1})})_+\wedge S^{(2b+6)\mtR}.
  \end{tikzcd}
\end{equation}

\subsectionlabel{Equivariant K-theory}
 The real representation $6\mtR$ can be identified with the complex representation $3\mtC$, and the real and complex representations of $\mH$ coincide. Therefore, we can apply $C_2$-equivariant reduced complex $\sfK$-theory to diagram (\ref{eqn:main-diagram-cofibers-spaces}) to get a long exact sequence in each row. In particular, with our map $_{S^1\setminus}(\calF\BF_{S^2,V})$, we get the following commutative diagram where the rows are exact

\begin{equation}\label{eqn:main-k-diagram}
  \begin{tikzcd}[ampersand replacement = \&]
    \tilde{\sfK}_{C_2}((E_{\mtCP^{2a+3}})_+\wedge S^{b\mtC})\& \tilde{\sfK}_{C_2}(S^{b\mtC})\ar[l,"p_+^\ast"']\&\tilde{\sfK}_{C_2}( (E_{S(\mtCP^{2a+3})})_+\wedge S^{b\mtC})\ar[l,"i_+^\ast"']\\
    \tilde{\sfK}_{C_2}((E_{\mtCP^{2a-1}})_+\wedge S^{(b+3)\mtC})\& \tilde{\sfK}_{C_2}(S^{(b+3)\mtC})\ar[l,"p_-^\ast"]\ar[u,"a_{\mtC}^3"]\& \tilde{\sfK}_{C_2}((E_{S(\mtCP^{2a-1})})_+\wedge S^{(b+3)\mtC} )\ar[l,"i_-^\ast"]\ar[u,"{(_{S^1\setminus}\calF\BF_{S^2,V})^\ast}"].
  \end{tikzcd}
\end{equation}
Here, $a^3_{\mtC}$ is the $\sfK$-theoretic Euler class induced by the inclusion $S^0\into S^{3\mtC} = S^{6\mtR}$. By exactness and commutativity, we should get that $p_+^\ast a_{\mtC}^3 i_-^\ast \equiv 0$. We will eventually proceed by calculating some of the rings in this diagram. Note that we can apply the $\sfK_{C_2}$-theoretic Euler class $a^b_{\mtC}$ to each group in the diagram to remove the copies of $S^{b\mtC}$.

\subsectionlabel{Desuspension}
Before we calculate those rings, we must note that this plan of action is dependent on the vector space $V$. As a quick aside, we show that we only need to look at the space level, i.e. when $a\mH = 0$, via a desuspension result. 

\begin{lemma}[Desuspension Lemma]{}\label{lem:desuspension}
  Let $V = a\mH+b\mtR$. Let $$f,g\in \text{Map}^G(S^1_+\wedge S^{2\mH+b\mtR}, S^{(6+b)\mtR})$$ such that $f|_{S^1_+\wedge S^{b\mtR}}= g|_{S^1_+\wedge S^{b\mtR}}$, and suppose $$\uno_{a\mH}\wedge f\simeq  \uno_{a\mH}\wedge g$$ relative to ${S^1_+\wedge S^{b\mtR}}$ in $\text{Map}^G(S^1_+\wedge S^{2\mH+V}, S^{6\mtR+V})$. Then $f\simeq g$. 
\end{lemma}

The proof of this result uses the methods of \cite[II.3]{Dieck+1987} and is inspired by the result in \cite[2.1]{BirgitSchmidt2003Spin4A}. The key facts used are from equivariant obstruction theory. The obstruction to lifting a map $f$ from the $n$-skeleton to the $n+1$ skeleton relative to the $n-1$ skeleton is $c^{n+1}(f)$ in the exact sequence \cite[II.3.10]{Dieck+1987} $$[X_{n+1},Y]\to \text{Image}([X_n,Y]\to [X_{n-1},Y])\xr{c^{n+1}}\sfH^{n+1}_G(X,A;\pi_nY).$$ The difference cochain between two such lifts is $$\gamma(f,g)\in \sfH^n(X,A;\pi_n Y)\xleftrightarrow{\text{bij.}}[X_n,Y]^G_{\text{rel }A}$$ in \cite[II.3.17]{Dieck+1987}.

\begin{proof}
We define the spaces $(X,A) := (S^1_+\wedge S^{2\mH+b\mtR}, S^1_+\wedge S^{b\mtR})$ and $(X_V,A_V):=(X\wedge S^{a\mH}, A\wedge S^{a\mH})$. Observe that $X$ is obtained from $A$ by attaching free $\Pin(2)$-cells of dimensions up to $\Pin(2)\by D^{8+b}$.  We set $(\hat{X},\hat{A}) := (X,A)\by (I,\dee I)$ and $(\hat{X}_V,\hat{A}_V) := (X_V,A_V)\by (I,\dee I)$. Written alternatively, this is 
\begin{align}
      (\hat X,\hat A)&:=(X\by I, (A\by I)\cup (X \by \dee I)), \text { and } \\
      (\hat X_V,\hat A_V)&:=(X_V\by I, (A_V\by I)\cup (X_V \by \dee I)).
\end{align}
We also define $Y := S^{(6+b)\mtR}$ and $Y_V := S^{6\mtR+a\mH+b\mtR}$. 

The hypothesis gives us maps $H^{(-1)}:\hat{A}\to Y$, $H_V:\hat X_V\to Y_V$, where $H^{(-1)} = f\sqcup g$ on $\dee I$ and is the stated homotopy along $A\by I$ and $H_V$ is the stated homotopy between $\uno_{a\mH}\wedge f$ and $\uno_{a\mH}\wedge g$. We would like to extend $H^{(-1)}$ to the entirety of $\hat{X}$ as this will prove our result. 

We will first need a lemma to transport information between $H^{(k)}$ and $H_V$. 
\begin{lemma}
    Let $k>0$, $n\leq 2\text{Conn}(Y)-1$, and $\pi_1(Y)\acts \pi_n(Y)$ trivially. Then there is a homomorphism
    $$\phi^{k,n}:\sfH^{k}_G(\hat{X},\hat{A};\pi_n Y)\xr{\cong} \sfH^{k+\abs{V}}_G(\hat X_V,\hat A_V;\pi_{n+\abs{V}}Y_V)$$ which is an isomorphism for $n\leq 2\text{Conn}(Y)-2$ and a surjection for $n = 2\text{Conn}(Y)-1$. 
\end{lemma}

For $k\geq 0$, $\sfH^k_G(X,A;G)$ is the cohomology group defined in \cite[II.3]{Dieck+1987}.

\begin{proof}
    Consider the sequence of homomorphisms 
    \begin{align*}
        \sfH^{k}_G(\hat{X},\hat{A};\pi_n Y)\cong \tilde{\sfH}^k_G(\hat{X}/\hat{A};\pi_n Y)\xrightarrow[\text{Thom}]{\cong}&\tilde{\sfH}^{k+\abs{V}}_G(\hat{X}/\hat{A}\wedge S^V;\pi_nY)\\
        &\cong \sfH^{k+\abs{V}}_G(\hat X_V,\hat A_V;\pi_nY)\xrightarrow{\text{Freud.}}\sfH^{k+\abs{V}}_G(\hat X_V,\hat A_V;\pi_{n+\abs{V}}Y_V),
    \end{align*}
    where we combine the natural isomorphism $\sfH^k(X,A)\cong \tilde{\sfH}^k(X/A)$ with the equivariant Thom isomorphism \cite{costenoble2013equivariantordinaryhomologycohomology,lewis1986equivariant}
    and the Freudenthal suspension theorem, which states $\pi_k(S^n)\cong \pi_{k+m}(S^{n+m})$ when $k\leq 2n-2$ and is a surjection when $k = 2n-1$. 
\end{proof}
Using this, we will proceed inductively. We are able to lift $H^{(-1)}$ from the $(-1)$-skeleton to $H^{(0)}$ on the $0$-skeleton since the gluing maps are simply taking the disjoint union with free $G$-cells of dimension 0. We can also show that $H^{(-1)}\wedge \uno_{a\mH}= H_V^{(\abs{a\mH}-1)}$. These two maps are equal to each other when restricted to $\hat{A}\cup (X_V\by \dee I).$  Notice $\hat{A}_V$ is obtained from $\hat{A}\cup (X_V\by \dee I)$ by attaching free cells of up to dimension $1+b+4a$. On the other hand, the target of these maps is $Y_V$, which is a sphere of dimension $6+4a+b$. Since $\pi_{k}(S^{6+4a+b}) = 0$ for any $k<6+4a+b$, there is no obstruction to the homotopy between the maps $H^{(-1)}\wedge \uno_{a\mH}$ and $H_V^{(\abs{a\mH}-1)}$. After applying the homotopy extension property, we may homotope $H_V$ such that $H^{(-1)}\wedge \uno_{a\mH}$ is actually equal to $H_V^{(\abs{a\mH}-1)}$. 

Suppose we have lifted to $H^{(k)}$ such that $H^{(k)}\wedge \uno_{a\mH}= H_V^{(k+\abs{a\mH})}$. We can lift $H^{(k)}$ to the $(k+1)$-skeleton if $c^{k+1}(H^{(k)}) = 0\in \sfH^{k+1}(\hat{X},\hat{A};\pi_{k}Y)$. By the existence of $H_V$, we know $H_V^{(k+1+\abs{V})}$ is a lift of $H_V^{(k+\abs{V})}$, so $c^{k+1+\abs{V}}(H_V^{k+\abs{V}})=0$, meaning there is a lift of $H^{(k)}$ by using our isomorphism $(\phi^{k+1,k})\inv$. Denote this lift by $H_0^{(k+1)}$. See  $$d:=d(H_0^{(k+1)}\wedge\uno_V,H_V^{(k+1+\abs{V})})\in \sfH^{k+1+\abs{V}}(\hat{X}_V,\hat{A}_V;\pi_{k+1+\abs{V}}Y_V).$$ Set $g:=(\phi^{k+1,k+1})\inv(d)\in \sfH^{k+1}(\hat{X},\hat{A};\pi_{k+1}Y)$. We choose $H_1^{(k+1)}$ such that $d(H_0^{(k+1)},H_1^{(k+1)}) = g$. Then $H_1^{(k+1)}\wedge \uno_V = H_V^{(k+1+\abs{V})}$, and we let $H^{(k+1)}$ be $H_1^{(k+1)}$. 

This induction works so long as the homomorphism from Freudenthal is surjective. Since our homomorphisms are isomorphisms for $k = 0, \dots, 8+b$ and a surjection at $k = 9+b$, this is still well-defined, and we have such lifts. 

Therefore, we have shown there is a map $H:(\hat{X},\hat{A})\to Y$ exhibiting a homotopy between $f$ and $g$. 
\end{proof}
This allows us to work specifically with $a\mH=0$, so we can work on the $\sfK$-theory in the diagram 

\begin{center}
  \begin{tikzcd}[ampersand replacement = \&]
    \tilde{\sfK}_{C_2}((E_{\mtCP^{3}})_+)\& \tilde{\sfK}_{C_2}(S^0)\ar[l,"p_+^\ast"']\&\tilde{\sfK}_{C_2}( (E_{S(\mtCP^{3})})_+)\ar[l,"i_+^\ast"']\\
    \tilde{\sfK}_{C_2}(S^{3\mtC})\& \tilde{\sfK}_{C_2}(S^{3\mtC})\ar[l,"p_-^\ast"]\ar[u,"a_{\mtC}^3"]\& \tilde{\sfK}_{C_2}(S^{3\mtC} )\ar[l,"i_-^\ast"]\ar[u,"{_{S^1\setminus}\calF\BF_{S^2,0}^\ast}"].
  \end{tikzcd}
\end{center}

\subsectionlabel{K-theory Calculations}
With that aside, we do calculations in $\sfK$-theory. First, we record the $\sfK$-theory of $\mRP^n$. For a nice treatment, see the course notes by Randall-Williams \cite[\S5.5]{Randall-Williams-Char-Class-K-theory-classnotes}. We record the groups from those notes on $\sfK$-theory groups of $\mRP^n$ for the reader. 
\begin{align*}
  \sfK^0(\mRP^{2n+1})&= \mZ\oplus \mZ/2^n&\sfK^1(\mRP^{2n+1})&= \mZ\\
  \sfK^0(\mRP^{2n})&= \mZ\oplus\mZ/2^n&\sfK^1(\mRP^{2n})&= 0.
\end{align*}

By the relationships in the short exact sequence $0\to S^1\to \Pintwo\to C_2\to 0,$ and properties of equivariant $\sfK$-theory, collected in \cite[\S2]{manolescu2014intersection} in reference to \cite{atiyah2018k,Segal}, we obtain the following lemma.

\begin{lemma}\label{lem:K theory of BPin(2)^4n-2}
  $\tilde{\sfK}_{\Pintwo}(S(n\mH)_+)\cong \sfR(\Pintwo)/(z^n) = \frac{\mZ[z,w]}{(z^n,zw-2w, w^2-2w) }$. 

  Consequently, this implies $\tilde{\sfK}_{C_2}(\mtCP^{2n-1}_+)\cong \tilde{\sfK}(B\Pintwo^{(4n-2)}_+) \cong \tilde{\sfK}_{\Pintwo}(S(n\mH)_+)$. As groups, this is of type $\mZ^n\oplus \mZ/2^n$ and as $\mZ$-modules, all of these groups in cohomological degree 0 are isomorphic to $$\mZ\sets{1,z,\dots, z^{n-1}}\oplus \mZ/2^n\sets{w}.$$ In particular, the torsion subgroup is cyclic of order $2^n$. 
\end{lemma}
\begin{proof}
    Notice that $S(\infty\mH)$ is an $E\Pintwo$, and the quotient by $\Pintwo$ is a $B\Pintwo$. By abusing notation, we will call $S(\infty\mH)/\Pintwo$ $B\Pintwo$. Then note that $S(n\mH)/\Pintwo$ is the $4n-2$ skeleton of $B\Pintwo$, which we denote $B\Pintwo^{(4n-2)}$. 
    Also note that from \cite[Fact 2.3]{manolescu2014intersection}, $\sfK(B\Pintwo^{(4n-2)})\xr{\sim}\sfK_{\Pintwo}(S(n\mH))$. 

    There is a cofiber sequence $$S^0\to S^{n\mH}\to \Sigma S(n\mH)_+.$$ Applying $\tilde{\sfK}_{\Pintwo}$-theory to this \cite[2.9]{manolescu2014intersection}, we get the long exact sequence 
    \begin{center}
      \begin{tikzcd}[ampersand replacement = \&]
        \tilde{\sfK}^0_{\Pintwo}(S^0)\ar[d]\& \tilde{\sfK}^0_{\Pintwo}(S^{n\mH})\ar[l]\&\tilde{\sfK}^0_{\Pintwo}(\Sigma S(n\mH)_+)\ar[l]\\
        \tilde{\sfK}^{-1}_{\Pintwo}(\Sigma S(n\mH)_+)\ar[r]\& \tilde{\sfK}^{-1}_{\Pintwo}(S^{n\mH})\ar[r]\& \tilde{\sfK}^{-1}_{\Pintwo}(S^0)\ar[u]
      \end{tikzcd}.
    \end{center}
    After applying \cite[2.8,2.14]{manolescu2014intersection} we get the information
    \begin{center}
        \begin{tikzcd}[ampersand replacement = \&]
          \sfR(\Pintwo)\ar[d]\& \sfR(\Pintwo)\ar[l,"z^n\cdot(-)"]\&\tilde{\sfK}^0_{\Pintwo}(\Sigma S(n\mH)_+)\ar[l]\\
          \tilde{\sfK}^{-1}_{\Pintwo}(\Sigma S(n\mH)_+)\ar[r]\& 0\ar[r]\& 0\ar[u]
        \end{tikzcd}.
      \end{center}
    Since the map $z^n\cdot(-)$ is injective and the map down on the left is surjective, and using \cite[2.8]{manolescu2014intersection} once again, we get $$\tilde{\sfK}^i_{\Pintwo}(S(n\mH)_+) = \begin{cases}
        \frac{\mZ[z,w]}{(zw-2w,w^2-2w,z^n)}&i\text{ even},\\
        0&i\text{ odd}\\
    \end{cases}.$$
    Looking just at the group structure, we get that the group structure of $$\sfK^0(B\Pintwo^{(4n-2)})\cong \mZ^{n}\oplus \mZ/2^n.$$
\end{proof}

With this lemma in hand, we can work through the calculations for the main diagram (\ref{eqn:main-k-diagram}). We wanted to calculate everything with $C_2$-equivariant $\sfK$-theory, but we can also calculate after taking the space-level quotient by the $C_2$-action on the left. This yields bundles $E_{B\Pintwo^{(6)}}$ over $S^2$. We can also form the bundle $E_{\mHP^1}\cong S^2\by \mHP^1$ since $\Sp(1)$ (and $\Sp(n)$, for that matter) is simply-connected. Now, there is a fiber bundle $\mRP^2\into E_{B\Pintwo^{(6)}}\downto E_{\mHP^1}\cong S^2\by \mHP^1$. Taking pullbacks along $S^2\by \ast\xr{\ell} S^2\by \mHP^1\xleftarrow{r} \ast\by \mHP^1$, we get bundles 
\begin{center}
  $\begin{tikzcd}[ampersand replacement = \&]
      \mRP^2\ar[r,hook] \& \ell^\ast E_{B\Pintwo^{(6)}}\ar[d]\ar[r,equal]\&[-15pt] \!\!:F_{S^2}\\
      \& S^2
    \end{tikzcd}$ and $\begin{tikzcd}[ampersand replacement = \&]
      \mRP^2\ar[r,hook] \&  r^\ast E_{B\Pintwo^{(6)}}\ar[d]\ar[r,equal]\&[-15pt] B\Pintwo^{(6)}\\
      \& \mHP^1.
    \end{tikzcd}$
\end{center} 

\begin{sseqdata}[grid = chess,name = EBPin2-fs2-fiber,cohomological Serre grading, yscale = .4]
\foreach \x in {0,4}{
    \foreach \y in {0,...,4}{
        \class (\x,{2*\y})
    }
}
\end{sseqdata}

\begin{sseqdata}[grid = chess,name = EBPin2-fhp1-fiber,cohomological Serre grading, yscale = .4,xscale = 1.5,classes = {draw = none}]
\foreach \x in {0,2}{
    \foreach \y in {0,...,4}{
        \class["\mZ^2\oplus \mZ/4"] (\x,{2*\y})
    }
}
\end{sseqdata}

\begin{sseqdata}[grid = chess,name = EBPin2-rp2-fiber,cohomological Serre grading, yscale = .4,xscale = 1.5,classes ={draw = none},x range = {-1}{6}]
\foreach \x in {0,2,4,6}{
    \foreach \y in {0,1,2,3,4}{
        \class["\mZ\oplus \mZ/2"] (\x,{2*\y})
    }
}
\end{sseqdata}

\begin{sseqdata}[grid = chess,name = fs2,cohomological Serre grading, yscale = .4,xscale = 1,classes ={draw = none},x range = {-1}{6}]
    \foreach \x in {0,2}{
        \foreach \y in {0,1,2,3}{
            \class["\mZ\oplus \mZ/2"] (\x,{2*\y})
        }
    }
\end{sseqdata}
\begin{sseqdata}[grid = chess,name = fhp1,cohomological Serre grading, yscale = .4,xscale = 1.5,classes ={draw = none},x range = {-1}{6}]
    \foreach \x in {0,4}{
        \foreach \y in {0,1,2,3}{
            \class["\mZ\oplus \mZ/2"] (\x,{2*\y})
        }
    }
\end{sseqdata}

Our goal is to calculate $\tilde{\sfK}((E_{B\Pintwo^{(6)}})_+)$ and we will use the Atiyah-Hirzebruch spectral sequence to do so. The form of this spectral sequence is given by $$\xE_2^{p,q} = \sfH^p(X;\sfK^q(F))\Rightarrow \sfK^{p+q}(E),$$ where $F\to E\downto X$ is a Serre fibration.

Calculating $\tilde{\sfK}((F_{S^2})_+)$ and $\tilde{\sfK}(B\Pintwo^{(6)}_+)$ and importing the calculations from their spectral sequences will make this easier. 

$\sfK(B\Pintwo^{(6)})$ is easy from Lemma \ref{lem:K theory of BPin(2)^4n-2}, giving $\tilde{\sfK}(B\Pintwo^{(6)}_+) \cong \mZ\oplus \mZ/4$. 

Putting this calculation with the spectral sequence below gives us the following $\text{E}_\infty$-page with the extension denoted by $b'$.

  \begin{equation}\label{eqn:sseq-bpin26}
  \raisebox{-1in}{\begin{sseqpage}[name = fhp1, page = 0, x range = {-1}{6}, y range = {-1}{6}]
      \class(0,8) \class(4,-2) \class (4,-4)
      \structline[thick,wongdarkblue,"b'"] (0,4) (4,0)
      \structline[thick,wongdarkblue] (0,6) (4,2)
      \structline[thick,wongdarkblue] (0,2) (4,-2)
      \structline[thick,wongdarkblue] (0,0) (4,-4)
      \structline[thick,wongdarkblue] (0,8) (4,4)
  \end{sseqpage}}
\end{equation}
\begin{center}
  AHSS for the $\sfK$-theory of the total space in the fibration $\mRP^2\into {B\Pintwo^{(6)}}\downto \mHP^1$. 
\end{center}

The bundle $\mRP^2\into F_{S^2}\downto S^2$ needs its own work. 
\begin{lemma}
  $F_{S^2}$ is stably equivalent to $S^2\vee \mRP^4$
\end{lemma}
\begin{proof} 
  Note that $E_{\mtCP^1}\subset E_{\mtCP^3}$ when viewing $\mCP^1$ as the subspace $$\sets{[\alpha:\beta:\gamma:\delta]\in \mCP^3\mid \gamma,\delta=0}.$$ Notice that $E_{\mtCP^1}$ is an $S^2$-bundle over $S^2$. We also see that, under the clutching map \begin{center}
    \begin{tikzcd}[ampersand replacement = \&]
      \ctheta:S^1\by \mtCP^1\ar[r]\& S^1\by \mtCP^1 \\ [-20pt]
      {(z,[a:b])}\ar[r,mapsto] \& {(z,[a:zb])}
    \end{tikzcd},
  \end{center} for each $z\in S^1$, there are two fixed points on $\mtCP^1$ under the map, namely $[1:0]$ and $[0:1]$. For the rest of the points in $\mtCP^1$, the clutching function is the same as the clutching map that defines the bundle ${\calO(1)}_{\mR}$ over $\mCP^1$. So together, this allows us to see that $E_{\mtCP^1}$ is the sphere bundle of the real vector bundle $S({\calO(1)}_{\mR}\oplus\ul{\mR})$ with a left $C_2$-action. 
  
  When we take the quotient out by the $C_2$-action, we get the bundle $E_{\mRP^2}$, the real projectivization of $\mP({\calO(1)}_{\mR}\oplus\ul{\mR})$. When we restrict on fibers to $\mRP^1$, i.e. $E_{\mRP^1}\subset E_{\mRP^2}$, we get a circle bundle over $S^2$ whose defining clutching map, since we quotiented by $C_2$, has degree 2, meaning that the first Euler class associated to this circle bundle is given by $2$. The bundle over $S^2$ with that class is $TS^2 = \mRP^3$, so we know that $_{C_2\setminus} E_{\mCP^3}$ contains $\mRP^3$ as a sub-skeleton.

  Using the Leray-Hirsch theorem, we get $\mZ/2$-cohomology of $\mP({\calO(1)}_{\mR}\oplus\ul{\mR})$ (as well as an incomplete picture of the Steenrod module structure of the cohomology in Remark \ref{rmk:steenrod-mod}). Grothendieck's characterization of Stiefel-Whitney classes tells us  \begin{equation}\sfH^\ast(\mP({\calO(1)}_{\mR}\oplus\ul{\mR});\mZ/2):=\frac{\sfH^\ast(S^2;\mZ/2)\anglebracket{1,\gamma,\gamma^2}}{\gamma^3 = w_3+w_2\gamma +w_1\gamma^2} = \frac{\mZ/2[s,\gamma]}{\gamma^3 = s\gamma}, ~~~\begin{matrix}
    \abs{s}=2\\
    \abs{\gamma} = 1
  \end{matrix}.\label{eqn:cohomology-of-projectivization}\end{equation} We know that $w_1({\calO(1)}_{\mR}\oplus\ul{\mR})$ is 0 since $\mR$ is trivial, and ${\calO(1)}_{\mR}$ is orientable. We also know that $w_2({\calO(1)}_{\mR}\oplus\ul{\mR}) =1$ by the Whitney product and the fact that $w_2({\calO(1)}_{\mR})$ is the mod 2 reduction of the first Chern class $c_1({\calO(1)}_{\mR}) = 1$. Lastly, $w_3({\calO(1)}_{\mR}\oplus\ul{\mR}) = 0$ since the dimension and the Whitney sum formula tell us that $w_3({\calO(1)}_{\mR}) = w_3(\ul{\mR}) = 0$. 

  Notice that there is a section $S^2\xr{\sigma} F_{S^2}$, where we pick $[0:0:1]\in \mRP^2$ for each fiber. This is possible since this is fixed under the clutching map $\ctheta$. Since there is a section, the projection map is a retract, which means  $F_{S^2}\simeq \text{Cof}(\sigma)\vee S^2$ stably \cite[2.4.21]{Malkiewich2023}. The cofiber of $\sigma$ is homotopy equivalent to $\Th(\text{Taut}(\mRP^3)) = \mRP^4$, meaning that our space is, in fact, stably homotopy equivalent to $S^2\vee \mRP^4$. 
\end{proof} 

\begin{remark}\label{rmk:steenrod-mod}
  One could also check that the Steenrod module structure of this space is the same as $\mRP^4\vee S^2$. This intuition is what led to the conclusion of the argument, so we include the ideas here. 

  The cohomology ring at equation (\ref{eqn:cohomology-of-projectivization}) gives us partial information about the Steenrod module structure, of which a picture is given below. 
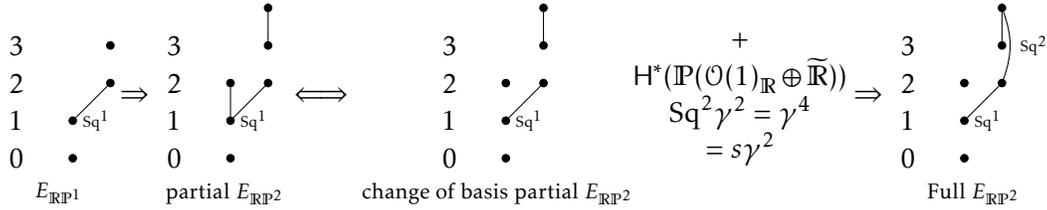
\begin{center}
    $\underset{E_{\mRP^1}}{\begin{tikzpicture}[main node/.style={circle,draw, minimum size = .1, inner sep = 1pt,fill},scale=.5]
          \node[main node] at (0,0)  (0) {};
          \node[main node] at (0,1)  (1) {};
          \node[main node] at (1,2)  (2a) {};
          \node[main node] at (1,3)  (3a) {};
          \draw (-1.5,0) node {$0$};
          \draw (-1.5,1) node {$1$};
          \draw (-1.5,2) node {$2$};
          \draw (-1.5,3) node {$3$};
          \path (1) edge[bend right = 0](2a) node[right] {\tiny$\Sq^1$};
  \end{tikzpicture}}\raisebox{1cm}{$\implies$}$ 
  $\underset{\text{ partial }E_{\mRP^2}}{\begin{tikzpicture}[main node/.style={circle,draw, minimum size = .1, inner sep = 1pt,fill},scale=.5]
          \node[main node] at (0,0)  (0) {};
          \node[main node] at (0,1)  (1) {};
          \node[main node] at (0,2)  (2) {};
          \node[main node] at (1,2)  (2a) {};
          \node[main node] at (1,3)  (3a) {};
          \node[main node] at (1,4)  (4a) {};
          \draw (-1.5,0) node {$0$};
          \draw (-1.5,1) node {$1$};
          \draw (-1.5,2) node {$2$};
          \draw (-1.5,3) node {$3$};
          \path (1) edge[bend right = 0] node [right] {} (2);
          \path (3a) edge[bend right = 0] node [right] {} (4a);
          \path (1) edge[bend right = 0](2a) node[right] {\tiny$\Sq^1$};
  \end{tikzpicture}}\raisebox{1cm}{$\iff $}$
  $\underset{\text{change of basis partial }E_{\mRP^2}}{\begin{tikzpicture}[main node/.style={circle,draw, minimum size = .1, inner sep = 1pt,fill},scale=.5]
          \node[main node] at (0,0)  (0) {};
          \node[main node] at (0,1)  (1) {};
          \node[main node] at (0,2)  (2) {};
          \node[main node] at (1,2)  (2a) {};
          \node[main node] at (1,3)  (3a) {};
          \node[main node] at (1,4)  (4a) {};
          \draw (-1.5,0) node {$0$};
          \draw (-1.5,1) node {$1$};
          \draw (-1.5,2) node {$2$};
          \draw (-1.5,3) node {$3$};
          \path (3a) edge[bend right = 0] node [right] {} (4a);
          \path (1) edge[bend right = 0](2a) node[right] {\tiny$\Sq^1$};
  \end{tikzpicture}}\raisebox{1cm}{$\begin{matrix}+\\\sfH^\ast(\mP({\calO(1)}_{\mR}\oplus\mtR))\\\Sq^2\gamma^2 = \gamma^4\\ = s\gamma^2\end{matrix}\implies$}$
  $\underset{\text{Full }E_{\mRP^2}}{\begin{tikzpicture}[main node/.style={circle,draw, minimum size = .1, inner sep = 1pt,fill},scale=.5]
          \node[main node] at (0,0)  (0) {};
          \node[main node] at (0,1)  (1) {};
          \node[main node] at (0,2)  (2) {};
          \node[main node] at (1,2)  (2a) {};
          \node[main node] at (1,3)  (3a) {};
          \node[main node] at (1,4)  (4a) {};
          \draw (-1.5,0) node {$0$};
          \draw (-1.5,1) node {$1$};
          \draw (-1.5,2) node {$2$};
          \draw (-1.5,3) node {$3$};
          \path (3a) edge[bend right = 0] node [right] {} (4a);
          \path (1) edge[bend right = 0](2a) node[right] {\tiny$\Sq^1$};
          \path (2a) edge[bend right = 20] node[right]{\tiny$\Sq^2$}(4a);
  \end{tikzpicture}}$

\end{center}
  The final picture gives us an idea that $E_{\mRP^2}$ could, in fact, be $S^2\vee \mRP^4$.
\end{remark}

\lineyspace

Since $\sfK$-theory is a stable-homotopy invariant, $\tilde{\sfK}(\mRP^4)= \mZ/4$, and $\tilde{\sfK}(S^2)=\mZ$, we get the following corollary.
\begin{corollary}
  $\tilde{\sfK}^i((F_{S^2})_+)\cong \tilde{\sfK}^i(S^2)\oplus \tilde{\sfK}^i(\mRP^4) = \begin{cases}
    \mZ\oplus \mZ/4 & i \text{ even}\\
    0 & i\text{ odd}
  \end{cases}$ 
\end{corollary}

Putting this corollary together with the AHSS for $F_{S^2}$, we get the following $\text{E}_\infty$-page with the extension denoted by $a'$. 

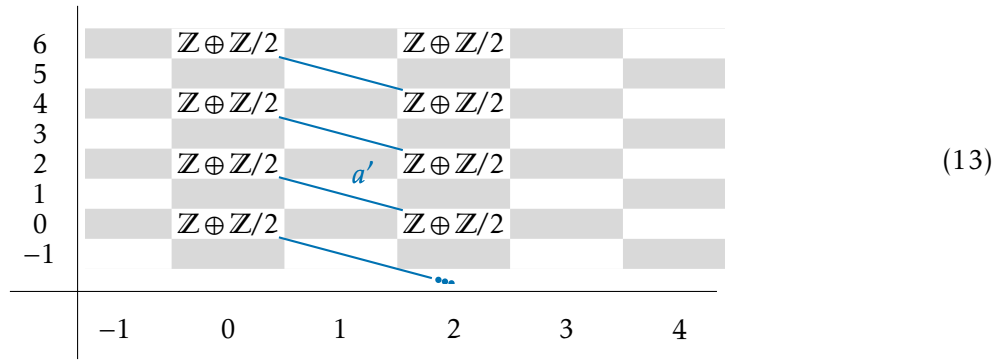
\begin{equation}\label{eqn:sseq-fs2}
    \raisebox{-1in}{\begin{sseqpage}[name = fs2, page = 0, x range = {-1}{4}, y range = {-1}{6},xscale = 1.5]
        \class (2,-2)
        \structline[thick,wongdarkblue,"a'"'] (0,2) (2,0)
        \structline[thick,wongdarkblue] (0,4) (2,2)
        \structline[thick,wongdarkblue] (0,6) (2,4)
        \structline[thick,wongdarkblue] (0,0) (2,-2)
    \end{sseqpage}}
\end{equation}
\begin{center}
  AHSS for the $\sfK$-theory of the total space in the fibration $\mRP^2\into F_{S^2}\downto S^2$.
\end{center}

With all of this we are able to calculate the $\sfK$-theory of $(E_{B\Pintwo^{(6)}})_+$. 
\begin{lemma}\label{lem:k-torsion-total-space}
  The torsion part of $\tilde{\sfK}((E_{B\Pintwo^{(6)}})_+)$ is isomorphic to $\mZ/8\oplus \mZ/2$. 
\end{lemma}
\begin{proof}
  Consider the bundle $\mRP^2\into E_{B\Pintwo^{(6)}}\downto E_{\mHP^1}\cong S^2\by \mHP^1$. The isomorphism is because $\pi_1(\Sp(1)) = 0$, meaning that any $\mHP^1$-bundle over $S^2$ is trivial. We can calculate the $\sfK$-theory via the Atiyah-Hirzebruch spectral sequence. For degree reasons, the spectral sequence collapses at $\xE_2$. We show the $\xE_2 = \xE_\infty$ page below. 
    \begin{equation}\label{eqn:sseq-ebpin26}
      \raisebox{-1in}{\begin{sseqpage}[name = EBPin2-rp2-fiber,page =0, x range = {-1}{7}, y range = {-1}{9}]
        \circleclasses[rounded rectangle](0,6)(6,0)
        \structline[bend right = 20,ultra thick, wongdarkblue,"b" very near end] (0,6) (4,2)
        \structline[bend left = 20,ultra thick, wongdarkblue,"a"] (0,6) (2,4)
        \structline[bend left = 20,ultra thick, wongdarkblue] (4,2) (6,0)
      \end{sseqpage}}
    \end{equation}

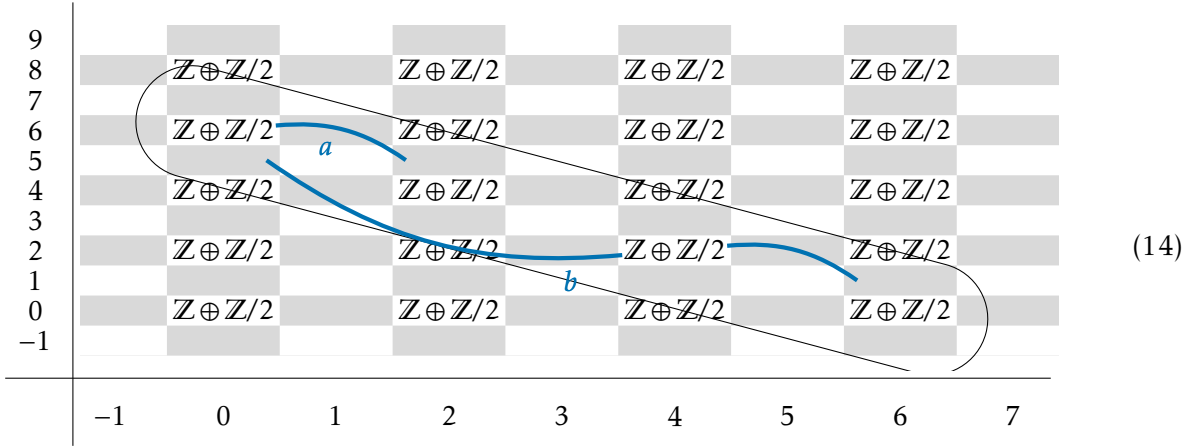
  \begin{center}
    $\sfH^x(\underbrace{S^2\by \mHP^1}_{E_{\mHP}};\sfK^y(\mRP^2))\implies \sfK^\ast(E_{B\Pintwo^{(6)}})$
  \end{center}
  From the inclusions $S^2\by \ast\xr{\ell} S^2\by \mHP^1$ and $\ast\by \mHP^1\xr{r} S^2\by \mHP^1$, there are inclusions $F_{S^2}\into E_{B\Pintwo^{(6)}}$ and $F_{\mHP^1}\into E_{B\Pintwo^{(6)}}$. Using the spectral sequences (\ref{eqn:sseq-bpin26}) and (\ref{eqn:sseq-fs2}), we can import the information into the spectral sequence (\ref{eqn:sseq-ebpin26})

  By functoriality of the Atiyah-Hirzebruch spectral sequence, this implies the two extensions in the diagram labeled as $a,b$. Since this $\xE_2$-page is both an $\sfH^\ast(F_{S^2};\sfK^\ast(\mRP^2))$-module and a $\sfH^\ast(F_{\mHP^1};\sfK^\ast(\mRP^2))$-module, we are able to import the extensions from the AHSS of $F_{S^2}$ and $F_{\mHP^1}\cong B\Pintwo^{(6)}$, which can be seen by comparing the different structures of $E_{B\Pintwo^{(6)}}$ to these different decompositions of $F_{S^2}$ and $F_{\mHP^1}$. If we pick the element $1$ in $\mZ/2$ at $(6,0)$ in the $\text{E}_\infty$ page, and follow the extension, we see that in $(4,2)$, the element is nonzero after multiplication by two. And in $(0,6)$, that same element is also nonzero after multiplication by 2 again. To be explicit, passing from $(6,0)$ through  the extensions to $(0,6)$, we include the structure of the extensions inductively. 
  \begin{align*}
    \mZ\oplus \mZ/2 \tag{(0,6)}\\
    \mZ^2\oplus \mZ/4 \tag{(0,6)-(4,2)}\\
    \mZ^3\oplus \mZ/4\oplus \mZ/2 \tag{(0,6)-(2,4)}\\
    \mZ^4\oplus \mZ/8\oplus \mZ/2 \tag{(0,6)-(6,0)}.
  \end{align*}
  Therefore, there must be an element of order 8, meaning $$\sfK^0(E_{B\Pintwo^{(6)}})\cong \mZ^4\oplus \mZ/8\oplus \mZ/2.$$
\end{proof}

\sectionlabel{Proof of Main Theorem}
With this in place, we can prove our main theorem. 
\begin{theorem*}
  The boundary Dehn twist is not isotopic to the identity in $\pi_0(\Diff_\dee(X^\circ))$. 
\end{theorem*}
\nond We reproduce diagram (\ref{eqn:main-k-diagram}) here for the reader with $V =0$. 
\begin{center}
  \begin{tikzcd}[ampersand replacement = \&]
    \tilde{\sfK}_{C_2}((E_{\mtCP^{3}})_+)\& \tilde{\sfK}_{C_2}(S^0)\ar[l,"p_+^\ast"']\&\tilde{\sfK}_{C_2}( (E_{S(\mtCP^{3})})_+)\ar[l,"i_+^\ast"']\\
    \tilde{\sfK}_{C_2}(S^{3\mtC})\& \tilde{\sfK}_{C_2}(S^{3\mtC})\ar[l,"p_-^\ast"]\ar[u,"a_{\mtC}^3"]\& \tilde{\sfK}_{C_2}(S^{3\mtC} )\ar[l,"i_-^\ast"]\ar[u,"{_{S^1\setminus}\calF\BF_{S^2,0}^\ast}"].
  \end{tikzcd}
\end{center}
\begin{proof}
  Suppose $\calF\BF_{S^2,V}$ does exist, and we desuspend it to $\calF\BF_{S^2,0}$ via Lemma \ref{lem:desuspension}. 
  Then the $C_2$-equivariant diagram above satisfies $p_+^\ast a_{\mtC}^3i_-^\ast\equiv 0$. 

  Lemma \ref{lem:k-torsion-total-space} gives the $\sfK_{{C_2}}$ groups as the direct sum of $\mZ$-modules $$\tilde{\sfK}_{C_2}((E_{\mtCP^3})_+) = \sfK(E_{B\Pintwo^{(6)}}) = \mZ^4\oplus \mZ/8\oplus \mZ/2.$$ The composition of maps $$\mtCP^1\into \mtCP^3\into E_{\mtCP^3}\to \ast$$ is identical to the map $\mtCP^1\to \ast$. Applying $\sfK_{C_2}$ to the map yields $$\sfR(C_2)=\sfK_{C_2}(\ast)\to \sfK_{C_2}(\mtCP^1) = \sfK(\mRP^2).$$ This map in $\sfK$-theory is surjective since $1$ is mapped to the trivial bundle over $\mRP^2$ and $\sigma$ is mapped to the complexification of the tautological bundle over $\mRP^2$. The second claim can be seen from the fact that, in $C_2$-spaces, we take the pullback 
  \begin{center}
    \begin{tikzcd}[ampersand replacement = \&]
      \mtCP^1\by \sigma\ar[r]\ar[dr,phantom,"\lrcorner"very near start]\ar[d]\& \sigma\ar[d]\\
      \mtCP^1\ar[r]\&\ast
    \end{tikzcd}.
  \end{center}
  let us view $\mtCP^1$ as $S^2$ with a free left $C_2$-action. If we take the quotient of $S^2\by \sigma$ by the $C_2$ action, there is a bundle isomorphism \begin{center}
    \begin{tikzcd}[ampersand replacement = \&]
      _{C_2\setminus} ({S}^2\by \sigma) \ar[r]\& \gamma_{\mRP^2}\otimes\mC \\ [-20pt]
      [(v,z)]\ar[r,mapsto] \& v\otimes z.
    \end{tikzcd}
  \end{center}
  This is well-defined since $(v,z)$ and $(-v,-z)$ are both sent to $v\otimes z$ by using the tensor product. We see  $\sfK(\mRP^2) = \mZ\oplus \mZ/2$, where the first summand is generated by the trivial bundle and the second summand is generated by the virtual bundle $\gamma_{\mRP^2}\oby \mC-1$.

  Therefore, the map in $\sfK_{C_2}$-theory is surjective. Putting this in the full composition, we get the three diagrams below which all contain the same information. 
  \begin{center}
    \fbox{\parbox{5in}{\begin{tikzcd}[ampersand replacement = \&]
      \sfK_{C_2}(\mtCP^1)\ar[from=r]\& \sfK_{C_2}(\mtCP^3)\ar[from=r] \& \sfK_{C_2}(E_{\mtCP^3})\ar[from=r] \&\sfK_{C_2}(\ast)
    \end{tikzcd}

      \begin{tikzcd}[ampersand replacement = \&]
        \sfK(\mRP^2)\ar[from=r]\& \sfK(B\Pintwo^{(6)})\ar[from=r] \& \sfK(E_{B\Pintwo^{(6)}})\ar[from=r] \&\sfK(\ast)\otimes \sfR(C_2)
      \end{tikzcd}

      \begin{tikzcd}[ampersand replacement = \&]
        \mZ\oplus \mZ/2\ar[from=r]\& \mZ^2\oplus\mZ/4\ar[from=r] \& \mZ^4\oplus \mZ/8\oplus \mZ/2\ar[from=r] \&\mZ^2
    \end{tikzcd} }}
  \end{center}

  The composition of maps $\sfK_{C_2}(\ast)\to \sfK_{C_2}(E_{\mtCP^3})\to \sfK_{C_2}(\mtCP^1)$ is surjective on the group on the right, so we hope to understand the map $$\sfK_{C_2}(E_{\mtCP^3}) = \sfK(E_{B\Pintwo^{(6)}})\to \sfK(\mRP^2) = \sfK_{C_2}(\mtCP^1).$$ Note that the skeleton $E_{B\Pintwo^{(6)}}$ can be filtered over the skeleta of $S^2\by \mHP^1$. The cells of the base are concentrated in degrees $0, 2, 4,$ and $6$. We see that column one of the spectral sequence (\ref{eqn:sseq-ebpin26}) corresponds to the filtration over the 0 skeleton, and is given by $\sfK(\mRP^2)$. The map from $\sfK(E_{B\Pintwo^{(6)}})\to \sfK(\mRP^2)$ can be found by inspection of the extensions in the spectral sequence. In particular, this means that the preimage of $\gamma_{\mRP^2}\oby \mC-1$ in $\sfK(E_{B\Pintwo^{(6)}})$ must be an element of order $8$.

  Now, since $a^6_{\mtR}$ corresponds to the Euler class $a^3_{\mtC}$, the composition $p^\ast a^3_{\mtC} = p^\ast w^3 = p^\ast 4w = 4p^\ast w$. However, by the properties outlined in the above paragraph, this is nonzero in the $\sfK(E_{B\Pintwo^{(6)}})$, which contradicts the fact it should be $0$. 
  
  Since we only assumed that the Dehn twist and the identity were smoothly isotopic in $\Diff_\dee(X)$, our assumption must have been faulty. Hence, the two maps are not smoothly isotopic. 
\end{proof}

\begin{corollary*}
  A smooth fiber bundle $\K3\#\K3\into E\downto S^2$ must have $w_2(T^vE)=0$.
\end{corollary*}
The proof of this corollary follows the proof of \cite[Prop 2.1]{KM20} with $\K3$ replaced by $\K3 \# \K3$.

  \newpage
\bibliographystyle{plain}
\bibliography{problemstatement}

\end{document}